\documentclass[11pt]{article}
\usepackage{etoolbox}
\usepackage{appendix}
\usepackage{subfig}
\usepackage{multirow}

\usepackage{amsmath}
\usepackage{amsfonts}
\usepackage[mathscr]{eucal}
\usepackage{amsthm}
\usepackage{amssymb}
\usepackage{cite}
\usepackage{enumerate}
\usepackage{tikz}\usetikzlibrary{matrix, calc, arrows, lindenmayersystems, decorations.pathreplacing, shapes.geometric}
\usepackage{hyperref} 
\usepackage{caption}
\usepackage{soul}

\usepackage[algo2e,linesnumbered]{algorithm2e}
\LinesNumbered 

\definecolor{myblue}{rgb}{0,0,0.6}     
\hypersetup{pdftitle={Numerical quadrature for singular integrals on fractals},  
     bookmarksopen=true, bookmarksopenlevel=4,
     colorlinks=true, linkcolor=myblue,  citecolor=myblue, filecolor=myblue,   urlcolor=myblue,  }

\usepackage{color}

\graphicspath{{./}{figs/}}

\DeclareMathOperator{\supp}{supp}

\DeclareMathOperator{\diam}{diam}
\DeclareMathOperator{\dist}{dist}

\usepackage{bbm}
\usepackage{color}

\newcommand{\Hull}{\mathrm{Hull}}

\begin{document}
\newcommand{\rf}[1]{(\ref{#1})}
\newcommand{\mmbox}[1]{\fbox{\ensuremath{\displaystyle{ #1 }}}}	%

\newcommand{\hs}[1]{\hspace{#1mm}}
\newcommand{\vs}[1]{\vspace{#1mm}}

\newcommand{\ri}{{\mathrm{i}}}
\newcommand{\re}{{\mathrm{e}}}
\newcommand{\rd}{\mathrm{d}}

\newcommand{\R}{\mathbb{R}}
\newcommand{\Q}{\mathbb{Q}}
\newcommand{\N}{\mathbb{N}}
\newcommand{\Z}{\mathbb{Z}}
\newcommand{\C}{\mathbb{C}}
\newcommand{\K}{{\mathbb{K}}}

\newcommand{\cA}{\mathcal{A}}
\newcommand{\cB}{\mathcal{B}}
\newcommand{\cC}{\mathcal{C}}
\newcommand{\cS}{\mathcal{S}}
\newcommand{\cD}{\mathcal{D}}
\newcommand{\cH}{\mathcal{H}}
\newcommand{\cI}{\mathcal{I}}
\newcommand{\cItilde}{\tilde{\mathcal{I}}}
\newcommand{\cIhat}{\hat{\mathcal{I}}}
\newcommand{\cIcheck}{\check{\mathcal{I}}}
\newcommand{\cIstar}{{\mathcal{I}^*}}
\newcommand{\cJ}{\mathcal{J}}
\newcommand{\cM}{\mathcal{M}}
\newcommand{\cP}{\mathcal{P}}
\newcommand{\cV}{{\mathcal V}}
\newcommand{\cW}{{\mathcal W}}
\newcommand{\scrD}{\mathscr{D}}
\newcommand{\scrS}{\mathscr{S}}
\newcommand{\scrJ}{\mathscr{J}}
\newcommand{\sD}{\mathsf{D}}
\newcommand{\sN}{\mathsf{N}}
\newcommand{\sS}{\mathsf{S}}
 \newcommand{\sT}{\mathsf{T}}
 \newcommand{\sH}{\mathsf{H}}
 \newcommand{\sI}{\mathsf{I}}

\newcommand{\bs}[1]{\mathbf{#1}}
\newcommand{\bb}{\mathbf{b}}
\newcommand{\bd}{\mathbf{d}}
\newcommand{\bn}{\mathbf{n}}
\newcommand{\bp}{\mathbf{p}}
\newcommand{\bP}{\mathbf{P}}
\newcommand{\bv}{\mathbf{v}}
\newcommand{\bx}{\mathbf{x}}
\newcommand{\by}{\mathbf{y}}
\newcommand{\bz}{{\mathbf{z}}}
\newcommand{\bxi}{\boldsymbol{\xi}}
\newcommand{\boldeta}{\boldsymbol{\eta}}	%

\newcommand{\ts}{\tilde{s}}
\newcommand{\tGamma}{{\tilde{\Gamma}}}
 \newcommand{\tbx}{\tilde{\bx}}
 \newcommand{\tbd}{\tilde{\bd}}
 \newcommand{\txi}{\xi}

\newcommand{\done}[2]{\dfrac{d {#1}}{d {#2}}}
\newcommand{\donet}[2]{\frac{d {#1}}{d {#2}}}
\newcommand{\pdone}[2]{\dfrac{\partial {#1}}{\partial {#2}}}
\newcommand{\pdonet}[2]{\frac{\partial {#1}}{\partial {#2}}}
\newcommand{\pdonetext}[2]{\partial {#1}/\partial {#2}}
\newcommand{\pdtwo}[2]{\dfrac{\partial^2 {#1}}{\partial {#2}^2}}
\newcommand{\pdtwot}[2]{\frac{\partial^2 {#1}}{\partial {#2}^2}}
\newcommand{\pdtwomix}[3]{\dfrac{\partial^2 {#1}}{\partial {#2}\partial {#3}}}
\newcommand{\pdtwomixt}[3]{\frac{\partial^2 {#1}}{\partial {#2}\partial {#3}}}
\newcommand{\bnabla}{\boldsymbol{\nabla}}
\newcommand{\dive}{\boldsymbol{\nabla}\cdot}
\newcommand{\curl}{\boldsymbol{\nabla}\times}
\newcommand{\Phixy}{\Phi(\bx,\by)}
\newcommand{\PhiOxy}{\Phi_0(\bx,\by)}
\newcommand{\dxPhixy}{\pdone{\Phi}{n(\bx)}(\bx,\by)}
\newcommand{\dyPhixy}{\pdone{\Phi}{n(\by)}(\bx,\by)}
\newcommand{\dxPhiOxy}{\pdone{\Phi_0}{n(\bx)}(\bx,\by)}
\newcommand{\dyPhiOxy}{\pdone{\Phi_0}{n(\by)}(\bx,\by)}

\newcommand{\eps}{\varepsilon}
\newcommand{\real}[1]{{\rm Re}\left[#1\right]} %
\newcommand{\im}[1]{{\rm Im}\left[#1\right]}
\newcommand{\ol}[1]{\overline{#1}}
\newcommand{\ord}[1]{\mathcal{O}\left(#1\right)}
\newcommand{\oord}[1]{o\left(#1\right)}
\newcommand{\Ord}[1]{\Theta\left(#1\right)}

\newcommand{\hsnorm}[1]{||#1||_{H^{s}(\bs{R})}}
\newcommand{\hnorm}[1]{||#1||_{\tilde{H}^{-1/2}((0,1))}}
\newcommand{\norm}[2]{\left\|#1\right\|_{#2}}
\newcommand{\normt}[2]{\|#1\|_{#2}}
\newcommand{\on}[1]{\Vert{#1} \Vert_{1}}
\newcommand{\tn}[1]{\Vert{#1} \Vert_{2}}
\newcommand{\xt}{\mathbf{x},t}
\newcommand{\PhiF}{\Phi_{\rm freq}}
\newcommand{\cone}{{c_{j}^\pm}}
\newcommand{\ctwo}{{c_{2,j}^\pm}}
\newcommand{\cthree}{{c_{3,j}^\pm}}

\newtheorem{thm}{Theorem}[section]
\newtheorem{lem}[thm]{Lemma}
\newtheorem{defn}[thm]{Definition}
\newtheorem{prop}[thm]{Proposition}
\newtheorem{cor}[thm]{Corollary}
\newtheorem{rem}[thm]{Remark}
\newtheorem{conj}[thm]{Conjecture}
\newtheorem{ass}[thm]{Assumption}
\newtheorem{example}[thm]{Example} %

\newcommand{\tH}{\widetilde{H}}
\newcommand{\Hze}{H_{\rm ze}} 	%
\newcommand{\uze}{u_{\rm ze}}		%
\newcommand{\dimH}{{\rm dim_H}}
\newcommand{\dimB}{{\rm dim_B}}
\newcommand{\IntClosOm}{\mathrm{int}(\overline{\Omega})}
\newcommand{\IntClosOmOne}{\mathrm{int}(\overline{\Omega_1})}
\newcommand{\IntClosOmTwo}{\mathrm{int}(\overline{\Omega_2})}
\newcommand{\Ccomp}{C^{\rm comp}}
\newcommand{\tCcomp}{\tilde{C}^{\rm comp}}
\newcommand{\uC}{\underline{C}}
\newcommand{\utC}{\underline{\tilde{C}}}
\newcommand{\oC}{\overline{C}}
\newcommand{\otC}{\overline{\tilde{C}}}
\newcommand{\capcomp}{{\rm cap}^{\rm comp}}
\newcommand{\Capcomp}{{\rm Cap}^{\rm comp}}
\newcommand{\tcapcomp}{\widetilde{{\rm cap}}^{\rm comp}}
\newcommand{\tCapcomp}{\widetilde{{\rm Cap}}^{\rm comp}}
\newcommand{\hcapcomp}{\widehat{{\rm cap}}^{\rm comp}}
\newcommand{\hCapcomp}{\widehat{{\rm Cap}}^{\rm comp}}
\newcommand{\tcap}{\widetilde{{\rm cap}}}
\newcommand{\tCap}{\widetilde{{\rm Cap}}}
\newcommand{\ccap}{{\rm cap}}
\newcommand{\ucap}{\underline{\rm cap}}
\newcommand{\uCap}{\underline{\rm Cap}}
\newcommand{\cCap}{{\rm Cap}}
\newcommand{\ocap}{\overline{\rm cap}}
\newcommand{\oCap}{\overline{\rm Cap}}
\DeclareRobustCommand
{\mathringbig}[1]{\accentset{\smash{\raisebox{-0.1ex}{$\scriptstyle\circ$}}}{#1}\rule{0pt}{2.3ex}}
\newcommand{\cirH}{\mathringbig{H}}
\newcommand{\cirHs}{\mathringbig{H}{}^s}
\newcommand{\cirHt}{\mathringbig{H}{}^t}
\newcommand{\cirHm}{\mathringbig{H}{}^m}
\newcommand{\cirHzero}{\mathringbig{H}{}^0}
\newcommand{\deO}{{\partial\Omega}}
\newcommand{\OO}{{(\Omega)}}
\newcommand{\Rn}{{(\R^n)}}
\newcommand{\Id}{{\mathrm{Id}}}
\newcommand{\gap}{\mathrm{Gap}}
\newcommand{\ggap}{\mathrm{gap}}
\newcommand{\isom}{{\xrightarrow{\sim}}}
\newcommand{\half}{{1/2}}
\newcommand{\mhalf}{{-1/2}}
\newcommand{\inter}{{\mathrm{int}}}

\newcommand{\Hsp}{H^{s,p}}
\newcommand{\Htq}{H^{t,q}}
\newcommand{\tHsp}{{{\widetilde H}^{s,p}}}
\newcommand{\SP}{\ensuremath{(s,p)}}
\newcommand{\Xsp}{X^{s,p}}

\newcommand{\dd}{{d}}\newcommand{\pp}{{p_*}}

\newcommand{\Rnn}{\R^{n_1+n_2}}
\newcommand{\Tr}{{\mathrm{Tr}}}

\renewcommand{\arraystretch}{1.7}
\renewcommand{\bs}[1]{\boldsymbol{#1}}
\newcommand{\vb}[1]{\vec{\bs{#1}}}
\newcommand{\be}{\bs{e}}
\renewcommand{\bn}{\bs{n}}
\renewcommand{\bx}{\mathbf{x}}
\renewcommand{\by}{\bs{y}}
\newcommand{\bg}{\bs{g}}
\newcommand{\bu}{\bs{u}}
\newcommand{\bw}{\bs{w}}
\newcommand{\bA}{\bs{A}}
\newcommand{\bC}{\bs{C}}
\newcommand{\bL}{\bs{L}}
\newcommand{\bS}{\bs{S}}
\newcommand{\bT}{\bs{T}}
\newcommand{\bU}{\bs{U}}
\newcommand{\bV}{\bs{V}}
\newcommand{\vbE}{\vb{E}}
\newcommand{\bX}{\bs{X}}
\newcommand{\bgamma}{{\bs{\gamma}}}
\newcommand{\bH}{\bs{H}}
\newcommand{\bnu}{\boldsymbol{\nu}}
\newcommand{\btau}{\boldsymbol{\tau}}
\newcommand{\bseta}{\boldsymbol{\eta}}
\newcommand{\UnionHull}{\Gamma_{\Hull,h}}
\newcommand{\UnionHullbm}{(\Gamma_{\bm})_{\Hull,h}}
\newcommand{\UnionHullbmp}{(\Gamma_{\bm'})_{\Hull,h}}
\newcommand{\cosc}{c_{\rm osc}}
\definecolor{purple0}{rgb}{0.4,0,0.5}
\definecolor{orange}{rgb}{1,0.4,0}
\newcommand{\cu}[1]{{\color{purple0} #1 }}
\definecolor{orange0}{rgb}{1,0.3,0}
\newcommand{\ctodo}[1]{{\color{orange0} {\bf TODO:} #1 }}
\definecolor{green0}{rgb}{0.1,0.6,0}
\newcommand{\dnote}[1]{{\color{green0} {\bf DH:} #1 }}

\allowdisplaybreaks[4]

\title{Numerical evaluation of singular integrals on non-disjoint self-similar fractal sets}

\author{A.\ Gibbs$^{\text{a},*}$,
		D.\ P.\ Hewett$^{\text{a}}$, 
		B. Major$^{\text{a}}$\\[2pt]
$^{\text{a}}${\footnotesize Department of Mathematics, University College London, London, United Kingdom}\\
${^*}${\footnotesize Corresponding author: andrew.gibbs@ucl.ac.uk}
}

\maketitle
\renewcommand{\thefootnote}{\arabic{footnote}}
\begin{abstract}
We consider the numerical evaluation of a class of double integrals with respect to a pair of self-similar measures over a self-similar fractal set {(the attractor of an iterated function system)}, with a weakly singular integrand of logarithmic or algebraic type. In a recent paper [Gibbs, Hewett and Moiola, Numer. Alg., 2023] it was shown that when the fractal set is ``disjoint'' in a certain sense (an example being the Cantor set), the self-similarity of the measures, combined with the homogeneity properties of the integrand, can be exploited to express the singular integral exactly in terms of regular integrals, which can be readily approximated numerically. In this paper we 
present a methodology for extending 
these results to cases where the fractal is non-disjoint   
{but non-overlapping (in the sense that the open set condition holds)}. 
Our approach applies to many well-known examples including the Sierpinski triangle, the Vicsek fractal, the Sierpinski carpet, and the Koch snowflake.



\end{abstract}

\newcommand{\bm}{{\mathbf{m}}}
\newcommand{\bmp}{{\mathbf{m}'}}
\renewcommand{\bn}{{\mathbf{n}}}
\newcommand{\bnp}{{\mathbf{n}'}}
\section{Introduction}
\label{sec:Intro}

In this paper we consider the numerical evaluation of integrals of the form
\begin{align}
\label{eq:integral}
I_{\Gamma,\Gamma} = \int_\Gamma \int_\Gamma \Phi_t(x,y)\,\rd\mu'(y)\,\rd\mu(x),
\end{align}
where (see \S\ref{sec:Prelim} for details) $\Gamma\subset\R^n$ is the attractor of an iterated function system {(IFS)} of contracting similarities satisfying the open set condition, $\mu$ and $\mu'$ are self-similar (also known as ``invariant'', or ``balanced'') measures on $\Gamma$, and 
\begin{align}
\label{eq:PhitDef}
\Phi_t(x,y):= \begin{cases}
|x-y|^{-t}, & t>0,\\
\log|x-y|, & t=0,
\end{cases}
\qquad x,y\in\R^n.
\end{align}

In the case where $t>0$ and $\mu'=\mu$ the integral \eqref{eq:integral} is known in the fractal analysis literature as the ``$t$-energy'', or ``generalised electrostatic energy'', of the measure $\mu$ (see e.g.~\cite[\S4.3]{Fal}, \cite[\S2.5]{Mattila15} and \cite[Defn 4]{mantica2007asymptotic}). 
Integrals of the form \eqref{eq:integral} also arise as the diagonal entries in the system matrix in Galerkin integral equation methods for the solution of PDE boundary value problems in domains with fractal boundaries, for instance in the scattering of acoustic waves by fractal screens \cite{HausdorffBEM}. In such contexts the accurate numerical evaluation of these matrix entries is crucial for the practical implementation of the methods in question. 

Numerical quadrature rules for \eqref{eq:integral} were presented recently in \cite{HausdorffQUAD} for the case where $\Gamma$ is disjoint (see \S\ref{sec:Prelim} for our definition of disjointness), e.g.\ a Cantor set in $\R$ or a Cantor dust in $\R^n$, $n\geq 2$. 
The approach of \cite{HausdorffQUAD} is to 
decompose $\Gamma$ 
into a finite union of self-similar subsets $\Gamma_1,\ldots,\Gamma_M$ (each similar to $\Gamma$) using the IFS structure, and to write $I_{\Gamma,\Gamma}$ as a sum of integrals over all possible pairs $\Gamma_m\times \Gamma_n$. 
Using the homogeneity properties of the integrand $\Phi_t(x,y)$, namely that $\Phi_t(x,y)=\tilde\Phi_t(|x-y|)$, where $\tilde\Phi_t(r):=r^{-t}$ for $t>0$ and $\tilde\Phi_t(r):=\log{r}$ for $t=0$, which satisfies, for $\rho>0$,
\begin{align}
\label{eq:PhitProp}
\tilde\Phi_t(\rho r)= \begin{cases}
\rho^{-t}\tilde\Phi_t(r), & t>0,\\
\log\rho + \tilde\Phi_t(r), & t=0,
\end{cases}
\end{align}
one can show that the ``self-interaction'' integrals over $\Gamma_m\times \Gamma_m$, for $m=1,\ldots,M$, can be expressed in terms of the original integral $I_{\Gamma,\Gamma}$, which allows $I_{\Gamma,\Gamma}$ to be written in terms of the integrals over $\Gamma_m\times \Gamma_n$, for $m,n=1,\ldots,M$, with $m\neq n$. When $\Gamma$ is disjoint the latter integrals are regular (i.e.\ they have smooth integrands), so that one can obtain a representation formula for the singular integral \eqref{eq:integral} (when it converges) as a linear combination of regular integrals, which can be readily evaluated numerically (see \cite[Thm~4.6]{HausdorffQUAD}, which generalises previous results for Cantor sets, e.g.\ \cite{bessis1987mellin}). 
In the non-disjoint case, however, %
distinct self-similar subsets of $\Gamma$ may be non-disjoint, intersecting at discrete points (such as for the Sierpinski triangle, see \S\ref{sec:Sierpinski}) or at higher dimensional 
sets (such as for the Sierpinski carpet, see \S\ref{sec:Carpet}, or the Koch snowflake, see \S\ref{sec:Snowflake}). This means that some of the integrals over $\Gamma_m\times \Gamma_n$, for $m\neq n$, are singular, reducing the accuracy of quadrature rules based on the representation formula of \cite[Thm~4.6]{HausdorffQUAD} (which assumes they are regular). In this paper we remedy this, showing that in many non-disjoint cases (including those mentioned above), by decomposing $\Gamma$ further into smaller self-similar subsets it is possible to find a finite number of ``fundamental'' singular integrals (including $I_{\Gamma,\Gamma}$ itself) that satisfy a square system of linear equations that can be solved to express $I_{\Gamma,\Gamma}$ purely in terms of regular integrals that are amenable to accurate numerical evaluation.

We will describe our methodology in generality in \S\ref{sec:Algorithm}, but for the benefit of the reader seeking intuition we illustrate here the basic idea for the simple 
case where $\Gamma=[0,1]^2\subset\R^2$ is the unit square, $\mu=\mu'$ is the Lebesgue measure on $\R^2$ restricted to $\Gamma$, and $t=1$. Despite not generally being regarded as a ``fractal'', the square $\Gamma$ can be viewed as the self-similar attractor of an iterated function system comprising four contracting similarities, and can accordingly be split into a ``level one'' decomposition of 4 squares of side length $1/2$, or a ``level two'' decomposition of 16 squares of side length $1/4$, as illustrated in Figure \ref{fig:Square}. Let $I_{m,n}$ denote the integral $\int_{\Gamma_m}\int_{\Gamma_n}|x-y|^{-1}\,\rd y\rd x$ where $\Gamma_m$ and $\Gamma_n$ are any of the level one squares in Figure \ref{fig:Square}.  
One can then express \eqref{eq:integral} as the sum $I_{\Gamma,\Gamma}=\sum_{m=1}^4\sum_{n=1}^4 I_{m,n}$ of 16 singular integrals over all the pairs of level one squares, which can be categorised as follows: 4 self interactions ($I_{1,1}$, $I_{2,2}$ etc.), 8 edge interactions ($I_{1,2}$, $I_{2,1}$, $I_{2,3}$ etc.) and 4 vertex interactions ($I_{1,3}$, $I_{3,1}$, $I_{2,4}$, {$I_{4,2}$}). By symmetry, each of the integrals in each category is equal to all the others, so that
\begin{align}
\label{eq:Square0}
I_{\Gamma,\Gamma} = 4I_{1,1} + 8 I_{1,2} + 4 I_{1,3}.
\end{align}
Furthermore, by \eqref{eq:PhitProp} (with $t=1$), combined with a change of variables, the self-interaction integral $I_{1,1}$ can be expressed in terms of the original integral, as 
\begin{align}
\label{eq:Square1}
 I_{1,1}=
 2\int_{\Gamma_1}\int_{\Gamma_1} |(2x)-(2y)|^{-1}\,\rd y\rd x
 =\frac{1}{8}I_{\Gamma,\Gamma}.
\end{align}
Combining this with \eqref{eq:Square0} we obtain the equation 
$\frac{1}{2}I_{\Gamma,\Gamma} = 8 I_{1,2} + 4 I_{1,3}$. 
To derive two further equations connecting $I_{\Gamma,\Gamma}$, $I_{1,2}$ and $I_{1,3}$ we move to the level two decomposition, extending our notation in the obvious way, so that e.g.~$I_{12,23}$ denotes the integral $\int_{\Gamma_{12}}\int_{\Gamma_{23}}|x-y|^{-1}\,\rd y\rd x$, where $\Gamma_{12}$ and $\Gamma_{23}$ are the level two squares labelled ``12'' and ``23'' in Figure \ref{fig:Square}. Then the edge interaction integral $I_{1,2}$ can be written as a sum of 16 integrals over pairs of level two squares, which, after applying symmetry simplifications, gives
\begin{align}
\label{eq:Square2}
I_{1,2} = 2I_{12,21} + 2 I_{12,24} + R_{1,2}, 
\end{align}
where $R_{1,2} = 4I_{11,21} + 4I_{11,24} + 2I_{11,22} + 2I_{11,23}$ is a sum of regular integrals. 
Similarly, the vertex interaction integral $I_{1,3}$ can be written as
\begin{align}
\label{eq:Square3}
I_{1,3} = I_{13,31}  + R_{1,3}=I_{12,24}  + R_{1,3},
\end{align}
where $R_{1,3} = 4I_{13,32} + 4I_{13,33} + 4I_{12,33} + 2I_{12,34} + I_{11,33}= 4I_{11,24} + 4I_{13,33} + 4I_{12,33} + 2I_{11,23} + I_{11,33}$ is a sum of regular integrals. 
Furthermore, by \eqref{eq:PhitProp}, combined with a change of variables, we have that 
\begin{align}
\label{eq:Square4}
I_{12,21} = \frac{1}{8}I_{1,2} \qquad \text{ and } \qquad I_{12,24} = \frac{1}{8}I_{1,3},
\end{align}
and inserting these identities into \eqref{eq:Square2} and \eqref{eq:Square3} gives our two sought-after equations, namely $\frac{3}{4}I_{1,2}=\frac{1}{4}I_{1,3}+R_{1,2}$ and $\frac{7}{8}I_{1,3}=R_{1,3}$. 

\begin{figure}[t!]
\centering
\includegraphics[width=0.75\textwidth]{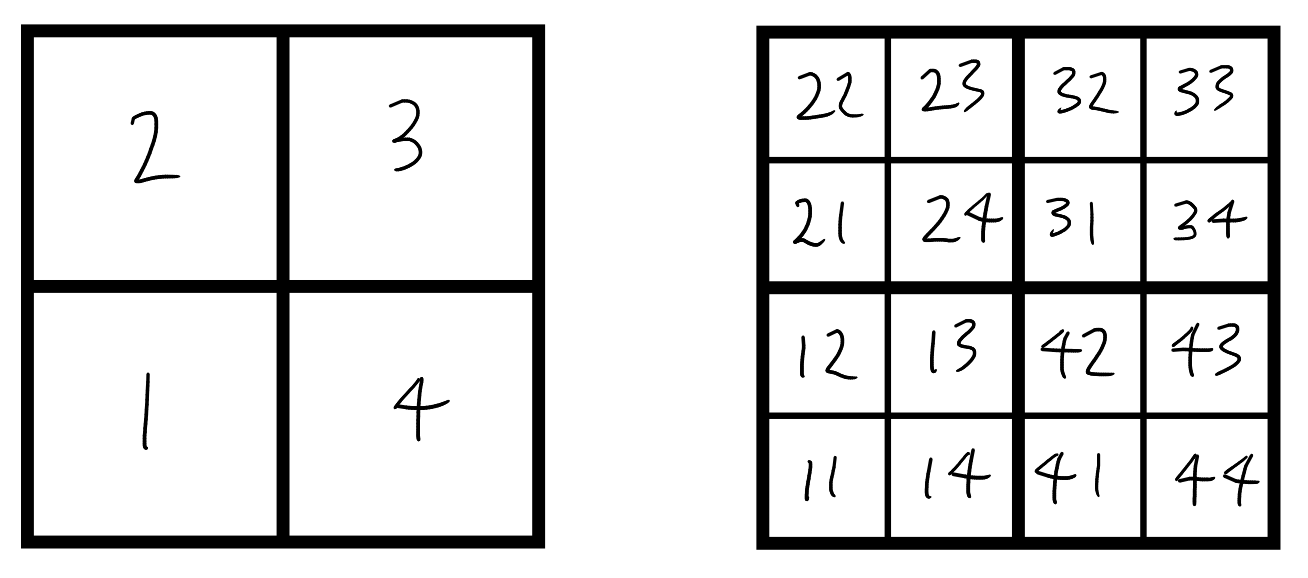}
\caption{Level 1 (left) and level 2 (right) subsets of the unit square. Here the labels ``$1$'' and ``$13$'' indicate the subsets $\Gamma_1$ and $\Gamma_{13}$ etc.\ referred to in the text.}
\label{fig:Square}
\end{figure}

To summarize, we have shown that the triple $(I_{\Gamma,\Gamma},I_{1,2},I_{1,3})^T$ satisfies the linear system
\begin{align}
\label{eq:LinearSystemSquare}
\left(\begin{array}{ccc}
\frac{1}{2} & -8 & -4 \\
0 & \frac{3}{4} & -\frac{1}{4} \\
0 & 0 & \frac{7}{8} 
\end{array}\right)
\left(\begin{array}{c}
I_{\Gamma,\Gamma}\\ 
I_{1,2}\\
I_{1,3}
\end{array}\right)
= 
\left(\begin{array}{c}
0\\ 
R_{1,2}\\
R_{1,3}
\end{array}\right),
\end{align}
and solving the system gives 
\begin{align}
\label{eq:SquareFinal}
I_{\Gamma,\Gamma} = \frac{64}{21}\left( 7R_{1,2}+5R_{1,3}\right),
\end{align}
which is an exact formula for $I_{\Gamma,\Gamma}$ in terms of the seven regular integrals $I_{11,21}$, $I_{11,24}$, $I_{11,22}$, $I_{11,23}$, $I_{13,33}$, $I_{12,33}$ and $I_{11,33}$, which are all amenable to accurate numerical evaluation, for instance with a product Gauss rule. 

Our goal in this paper is to derive formulas analogous to \eqref{eq:LinearSystemSquare} and \eqref{eq:SquareFinal} for 
more general $\Gamma$, $t$, $\mu$ and $\mu'$. 
The structure of the paper is as follows. In \S\ref{sec:Prelim} we review some preliminaries concerning self-similar fractal sets and measures. In \S\ref{sec:Similarity} we introduce the notion of ``similarity'' for integrals over pairs of subsets of $\Gamma$, and provide sufficient conditions under which it holds. In \S\ref{sec:Algorithm} we describe a general algorithm for generating linear systems of the form \eqref{eq:LinearSystemSquare} using our notion of similarity.
In \S\ref{sec:examples} we apply our algorithm to a number of examples including the Sierpinski triangle, the Sierpinski carpet, and the Koch snowflake. Finally, in \S\ref{sec:Numerics} we show how our results can be combined with numerical quadrature to compute accurate numerical approximations to the integral \eqref{eq:integral} in these and other cases. As {an} application, we show how our algorithm can be used in the context of the ``Hausdorff boundary element method'' of  \cite{HausdorffBEM} to compute acoustic scattering by non-disjoint fractal screens.

{Regarding related literature, we note that a three-dimensional version of the approach described above for integration over the unit square was used to compute the gravitational force between two cubes sharing a common face in \cite{trefethen2011ten}. More generally, this sort of approach forms the basis of the ``hierarchical quadrature'' developed for singular integrals over cubical and simplicial domains by B\"orm and Hackbusch \cite{borm2005hierarchical} and Meszmer \cite{meszmer2010hierarchical,meszmer2012hierarchical}.  
In the context of integration over fractals, we mention the work of Mantica \cite{Ma:96} and Strichartz \cite{strichartz2000evaluating}, where self-similarity techniques were used to derive exact formulas for integrals of polynomials over fractals. Our previous paper \cite{HausdorffQUAD} and the current paper can be viewed as extensions of the results of \cite{Ma:96} and  \cite{strichartz2000evaluating} to singular integrands.
}

\section{Preliminaries}
\label{sec:Prelim}

Throughout the paper we assume that $\Gamma$ is the 
attractor of an iterated function system (IFS) of contracting similarities satisfying the open set condition (OSC), 
meaning that (see e.g.\ \cite{hutchinson1981fractals})
\begin{enumerate}[(i)]
\item there exists $M\in\N$, $M\geq 2$, and a collection of maps $\{s_1,s_2,\ldots,s_M\}$, such that, for each $m=1,\ldots,M$, $s_m:\R^n\to\R^n$ satisfies
\begin{equation*}%
|s_m(x)-s_m(y)| = \rho_m|x-y|,\quad\text{for }x,y\in \R^{n},
\end{equation*} 
for some $\rho_m\in (0,1)$. 
Explicitly, for each $m=1,\ldots,M$ we can write  %
\begin{equation}\label{eq:affine}
s_m(x) = \rho_mA_mx + \delta_m,
\end{equation}
for some orthogonal matrix $A_m\in\R^{n\times n}$ and some translation $\delta_m\in\R^n$; 
\item $\Gamma$ is the unique non-empty compact set such that 
\begin{equation*}
\Gamma = s(\Gamma), 
\end{equation*}
where 
\begin{align}
\label{eq:fullmap}
s(E) := \bigcup_{m=1}^M  s_m(E), \quad  E\subset \R^{n};
\end{align}
\item there exists a non-empty bounded open set $O\subset \R^{n}$ such that
\begin{align} \label{oscfirst}
s(O) \subset O \quad \mbox{and} \quad s_m(O)\cap s_{m'}(O)=\emptyset, \quad m\neq m'\in\{1,\ldots,M\}.
\end{align}
\end{enumerate}

Then $\Gamma$ has Hausdorff dimension $\dimH(\Gamma)=d$, where $d\in (0,n]$ is the solution of the equation
\begin{align} \label{eq:dfirst}
\sum_{m=1}^M  \rho_m^d = 1.
\end{align}
We say that the IFS is \emph{disjoint} 
(cf.\ \cite[Defn 7.1]{barnsley2013developments})
if 
\begin{equation}\label{eq:Rdef}
\min_{m\neq m'}\big\{\dist \big(  s_{m}(\Gamma), s_{m'}(\Gamma)\big)  \big\}>0,
\end{equation}
which holds if and only if the open set $O$ in \eqref{oscfirst} can be taken such that $\Gamma\subset O$ 
\cite{HausdorffBEM}.

We say that the IFS is \textit{homogeneous} if $\rho_m=\rho$, $m=1,\ldots,M$, for some $\rho\in(0,1)$. In this case the solution of \eqref{eq:dfirst} is $d = \log{M}/\log{(1/\rho)}$.

To describe the decomposition of $\Gamma$ into self-similar subsets via the IFS structure we adopt the standard index notation of \cite{hutchinson1981fractals}. For $\ell\in \N$ let ${\Sigma}_\ell:=\{1,\ldots,M\}^\ell$, and, for $\bm=(m_1,\ldots,m_\ell)\in {\Sigma}_\ell$, let
\begin{equation*}%
\Gamma_{\bm} := s_{\bm}(\Gamma),\qquad \text{ where }\quad s_{\bm}:=s_{m_1}\circ\ldots\circ s_{m_\ell}.
\end{equation*}
When working with examples we shall typically write $\Gamma_{(m_1,m_2,\ldots,m_\ell)}$ as simply $\Gamma_{m_1m_2\ldots m_\ell}$ to make the notation more compact (as we did in \S\ref{sec:Intro}). 

For each $\bm=(m_1,\ldots,m_\ell)\in {\Sigma}_{\ell}$, $s_\bm$ is a contracting similarity of the form
\begin{align}
\label{eq:smExpression}
s_\bm(x) = \rho_\bm A_\bm x + \delta_\bm,
\end{align}
where (with the convention that an empty product equals $1$)
\[ \rho_\bm:=\prod_{i=1}^{\ell}\rho_{m_i}, \qquad A_\bm = \prod_{i=1}^{\ell}A_{m_i}, \qquad \delta_\bm = \sum_{i=1}^\ell \left(\prod_{j=1}^{i-1}\rho_{m_j}A_{m_j}\right)\delta_{m_i}.\]
The inverse of $s_\bm$ then has the representation\begin{align}
\label{eq:smInverseExpression}
s_\bm^{-1}(x) = \rho_\bm^{-1} A_\bm^{-1}(x - \delta_\bm).
\end{align}
Setting ${\Sigma}_{{\emptyset}}:=\{{\emptyset}\}$, $\Gamma_{{\emptyset}}:=\Gamma$ and $s_{{\emptyset}}(x):=x$, we define ${\Sigma}:={\Sigma}_{{\emptyset}} \cup (\cup_{\ell=1}^\infty {\Sigma}_\ell)$.

Given such a $\Gamma$, and a collection $(p_1,\ldots,p_M)$ of positive weights (or ``probabilities'') satisfying 
\begin{align}
\label{eq:weightsum}
0<p_m<1, \quad m=1,\ldots,M, \qquad \text{and} \qquad \sum_{m=1}^M p_m =1,
\end{align}
there exists \cite[Sections 4 \& 5]{hutchinson1981fractals} a positive Borel-regular finite measure $\mu$ supported on $\Gamma$, unique up to normalisation, called a \textit{self-similar} 
\cite{moran1998singularity} (also known as \textit{invariant} \cite{hutchinson1981fractals} or \textit{balanced} \cite{barnsley1985iterated}) measure, 
such that $\mu(E)=\sum_{m=1}^M p_m \mu(s_m^{-1}(E))$ for every measurable set $E\subset \mathbb{R}^n$. 
For such a measure, by \cite[Thm.~2.1]{moran1998singularity} 
the OSC \eqref{oscfirst} implies that $\mu(s_m(\Gamma)\cap s_{m'}(\Gamma))=0$ for each $m\neq m'$, and as a consequence we find that for $\bm=(m_1,\ldots,m_\ell)\in {\Sigma}$, and any $\mu$-measurable function $f$,\footnote{{A key step in the proof of this result is showing that $\mu(\Gamma_{m'}) = p_{m'}\mu(\Gamma)$ for ${m'}=1,\ldots,M$. To prove the latter, we first note that, by \eqref{eq:weightsum} and the positivity and self-similarity of $\mu$, if $E$ is measurable and $\mu(E)=0$ then $\mu(s_m^{-1}(E))=0$ for $m=1,\ldots,M$. In particular, since $\mu(s_m(\Gamma)\cap s_{{m'}}(\Gamma))=0$ for $m\neq {m'}$, we have that $\mu(s_m^{-1}(\Gamma_m\cap \Gamma_{m'}))=0$ for $m\neq {m'}$. Then, using the fact that $\mu$ is supported in $\Gamma$, we have $\mu(\Gamma_{m'})=\sum_{m} p_m \mu(s_m^{-1}(\Gamma_{m'})) = \sum_{m} p_m \mu(\Gamma\cap s_m^{-1}(\Gamma_{m'})) = \sum_{m} p_m \mu(s_m^{-1}(\Gamma_m\cap \Gamma_{m'})) = p_{m'} \mu(\Gamma)$, as claimed.}}
\begin{align}
\int_{\Gamma_{\bm}}f(x)\,\rd\mu(x) 
=p_\bm\int_\Gamma f\big(s_{\bm}(x)\big) \,\rd\mu(x),\label{eq:cov_general}
\end{align}
where (again with the convention that an empty product equals $1$)
\[ p_\bm:=\prod_{i=1}^{\ell}p_{m_i} .\]
In particular,
\begin{equation}\label{eq:sim_measure_general}
\quad \mu(\Gamma_{\bm})
= p_\bm\,\mu(\Gamma).
\end{equation}

\begin{example}
\label{ex:Hausdorff}
By choosing $p_m=\rho_m^d$ for $m=1,\ldots,M$ and normalising appropriately, we can obtain $\mu=\cH^d|_\Gamma$, where $\cH^d$ is the $d$-dimensional Hausdorff measure on $\R^n$ (note that {in this case} \eqref{eq:weightsum} holds by \eqref{eq:dfirst}). 
{We recall that $\cH^n$ is proportional to $n$-dimensional Lebesgue measure \cite[\S3.1]{Fal}.} 
\end{example}

{Given an IFS attractor $\Gamma$ and two self-similar measures $\mu$ and $\mu'$, with associated weights $(p_1,\ldots,p_M)$ and $(p'_1,\ldots,p'_M)$,} 
we define $t_*>0$, if it exists, to be the largest positive real number such that the integral $I_{\Gamma,\Gamma}$ converges for $0\leq t< t_*$. In \cite[Lem.~A.4]{HausdorffQUAD} we showed that if $\Gamma$ is disjoint then $t_*$ exists and is the unique positive real solution of the equation
\begin{align}
\label{eqn:tstar}
\sum_{m=1}^M p_mp_m' \rho_m^{-t_*}=1.
\end{align}
Our conjecture is that the same holds for non-disjoint $\Gamma$, under the assumption of the OSC \eqref{oscfirst}. As yet, we have not been able to prove this conjecture in its full generality. However, we shall proceed under the assumption that it holds, noting that it is well known to hold, with $t_*=d$, in the special case where $\mu=\mu'=\cH^d|_\Gamma$ for $d=\dimH(\Gamma)$ (see e.g. \cite[Corollary A.2]{HausdorffQUAD}), which is the case of relevance for the integral equation application from \cite{HausdorffBEM} that we study in \S\ref{sec:Numerics}. 
{We comment that when $\mu=\mu'$ the quantity $t_*$ is sometimes referred to as the ``electrostatic correlation dimension'' of $\mu$ \cite[Defn~6]{mantica2007asymptotic}. This and related notions of the ``dimension'' of a measure $\mu$ give, amongst other things, lower bounds on the Hausdorff dimension of the support of $\mu$ (which may be strictly smaller than that of $\Gamma$), and important information about the asymptotic behaviour of the Fourier transform of $\mu$ - see e.g.\cite{strichartz1990self}. 
}

We note that if the IFS is homogeneous then \eqref{eqn:tstar} can be solved analytically to give
\begin{align}
\label{eqn:tstar1}
t_* = \log\left(\sum_{m=1}^M p_mp_m'\right)/\log{\rho},
\end{align}
which reduces to $t_*=d$ in the case $\mu=\mu'=\cH^d|_\Gamma$ (where $p_m=p_m'=\rho^d$). 

{Self-similar measures sometimes possess useful symmetry properties.  
Let $T:\R^n\to\R^n$ be an isometry, 
with $T(x)=A_Tx+\delta_T$ for some orthogonal matrix $A_T$ and some translation vector $\delta_T$. We say that $\mu$ is \textit{invariant} under $T$ if $\mu(T(E))=\mu(E)$ for all measurable $E\subset\R^n$. If $\mu$ is invariant under $T$ then $\mu$ is also invariant under $T^{-1}$, the push-forward measure $\mu\circ T^{-1}:E\mapsto \mu(T^{-1}(E))$ coincides with $\mu$, and $T(\Gamma)=\Gamma$, so that by \cite[Thm.~3.6.1]{Bogachev} we have that, for all measurable $f$,}
\begin{align}
\label{eq:COV}
\int_\Gamma f(T(x))\rd\mu(x) = \int_\Gamma f(x)\rd\mu(x).
\end{align}

\begin{rem}
\label{rem:Isometries}
{Determining a complete list of isometries $T$ under which a given self-similar measure $\mu$ is invariant, directly from the IFS $\{s_1,\ldots,s_M\}$ and weights $p_1,\ldots,p_M$, appears to be an open problem. However, for specific examples it is often straightforward to determine the admissible $T$, as we demonstrate in \S\ref{sec:examples} and \S\ref{sec:Numerics}.
In the case $\mu=\cH^d|_\Gamma$ (see \S\ref{sec:Sierpinski}-\S\ref{sec:Snowflake}), a necessary and sufficient condition for $\mu$ to be invariant under $T$ is that $T(\Gamma)=\Gamma$, because $\cH^d$ is invariant under isometries of $\R^n$. For $\mu\neq\cH^d|_\Gamma$ it is still necessary that $T(\Gamma)=\Gamma$, but no longer sufficient (see \S\ref{sec:Interval}). 
Generically, a self-similar measure $\mu$ may not be invariant under any non-trivial isometries $T$ (see \S\ref{sec:nvo}). 
}
\end{rem}

\section{Similarity}
\label{sec:Similarity}

{We assume henceforth that $\Gamma$ is the attractor of an IFS $\{s_1,\ldots,s_M\}$ of contracting similarities satisfying the OSC, and that $\mu$ and $\mu'$ are self-similar measures on $\Gamma$, with associated weights $(p_1,\ldots,p_M)$ and $(p'_1,\ldots,p'_M)$. }

As mentioned in \S\ref{sec:Intro}, our approach to deriving representation formulas for \eqref{eq:integral} will be based on decomposing the integral $I_{\Gamma,\Gamma}$ into sums of integrals over pairs of subsets of $\Gamma$. For any two vector indices $\bm,\bm'\in {\Sigma}$ (possibly of different lengths) we define the \textit{sub-integral}
\[ I_{\bm,\bm'}:=\int_{\Gamma_\bm}\left(\int_{\Gamma_{\bm'}}\Phi_t(x,y)\rd\mu'(y)\right)\rd\mu(x). \]
Note that the original integral $I_{\Gamma,\Gamma}={I_{\emptyset,\emptyset}}$ is included in this definition. 
Central to our approach will be identifying when two sub-integrals $I_{\bm,\bm'}$ and $I_{\bn,\bn'}$ are \textit{similar}, in the sense that 
\begin{align}
\label{eq:SimilarityDefn}
I_{\bm,\bm'}=a I_{\bn,\bn'}+b, 
\end{align}
for some $a>0$ and $b\in\R$ that we can determine explicitly in terms of the parameters of the IFS and the measures $\mu$ and $\mu'$. 

The simplest instance of similarity occurs when $\mu=\mu'$, in which case the symmetry of the integrand (i.e.\ the fact that $\Phi_t(x,y)=\Phi_t(y,x)$), combined with Fubini's theorem, provides the following elementary result (which was used in the example in \S\ref{sec:Intro}, for which $I_{1,2}=I_{2,1}$ etc.).
\begin{prop}
\label{lem:Fubini}
If $\mu=\mu'$ then $I_{\bm,\bn}=I_{\bn,\bm}$ for each $\bm,\bn\in {\Sigma}$. 
\end{prop}

Other instances of similarity may be associated with the IFS structure (as for the derivation of \eqref{eq:Square1} and \eqref{eq:Square4}, in the example in \S\ref{sec:Intro}), and/or with other geometrical symmetries (as for the observation that $I_{1,2}=I_{2,3}$ etc., in the example in \S\ref{sec:Intro}). 
The following result provides sufficient conditions under which a given pair of sub-integrals $I_{\bm,\bm'}$ and $I_{\bn,\bn'}$ are similar in this manner. 
{We recall that the notion of a self-similar measure being invariant under an isometry $T$ was defined in \S\ref{sec:Prelim}, and that the question of determining for which isometries this holds was discussed in Remark \ref{rem:Isometries}. 
If no non-trivial isometries can be determined for $\mu$ or $\mu'$ one can always take $T$ and $T'$ to be the identity in the following. 
}
\begin{prop}
\label{lem:SimilaritySufficient}
Let $\bm,\bm',\bn,\bn'\in {\Sigma}$. %
Let $T$ and $T'$ be isometries of $\R^n$ such that $\mu$ and $\mu'$ are invariant under $T$ and $T'$ respectively. 
Suppose there exists $\varrho>0$ such that 
\begin{align}
\label{eq:SimilarityCond}
|s_{\bm}(T(s_{\bn}^{-1}(x)))-s_{\bm'}(T'(s_{\bn'}^{-1}(y)))| = \varrho|x-y|, \quad  x,y \in \R^n.
\end{align}
Then  
\begin{align}
\label{eq:Proportional}
I_{\bm,\bm'} = 
\begin{cases}
\displaystyle\frac{p_\bm p_{\bm'}'}{p_\bn p_{\bn'}'}\varrho^{-t} I_{\bn,\bn'}, & t\in(0,t_*),\\[4mm]
\displaystyle p_\bm p_{\bm'}'\left(\mu(\Gamma)\mu'(\Gamma)\log{\varrho} + \frac{1}{p_\bn p_{\bn'}'}I_{\bn,\bn'}\right) , & t=0.
\end{cases}
\end{align}
\end{prop}

\begin{proof}
By \eqref{eq:cov_general}, and the invariance of $\mu$ and $\mu'$ with respect to $T$ and $T'$, we have that 
\begin{align}
\notag
I_{\bm,\bm'} &= p_\bm p_{\bm'}'\int_\Gamma\int_\Gamma \Phi_t(s_\bm(x),s_{\bm'}(y))\rd\mu'(y)\rd\mu(x)\\
\label{eq:mInv}
&= p_\bm p_{\bm'}'\int_\Gamma\int_\Gamma \Phi_t(s_\bm(T(x)),s_{\bm'}(T'(y)))\rd\mu'(y)\rd\mu(x),\\
\label{eq:nInv}
I_{\bn,\bn'} &= p_\bn p_{\bn'}'\int_\Gamma\int_\Gamma \Phi_t(s_\bn(x),s_{\bn'}(y))\rd\mu'(y)\rd\mu(x).
\end{align}
Condition \eqref{eq:SimilarityCond} implies that 
\[
|s_\bm(T(x))-s_{\bm'}(T'(y))| = \varrho |s_\bn(x)-s_{\bn'}(y)|, \quad  x,y \in \R^n,
\]
which in turn, by \eqref{eq:PhitProp}, implies that 
\[
\Phi_t(s_\bm(T(x)),s_{\bm'}(T'(y)))=
\begin{cases}
\varrho^{-t} \Phi_t(s_\bn(x),s_{\bn'}(y)), & t\in(0,t_*),\\
\log{\varrho} + \Phi_t(s_\bn(x),s_{\bn'}(y)), & t=0,
\end{cases}
\qquad  x,y\in\R^n,
\]
and combining this with \eqref{eq:mInv}-\eqref{eq:nInv} gives the result. 
\end{proof}

The following result provides an equivalent characterization of the sufficient condition \eqref{eq:SimilarityCond} in terms of the scaling factors, orthogonal matrices and translations associated with the maps $s_\bm,s_\bmp,s_\bn,s_\bnp,T$ and $T'$ (see \eqref{eq:smExpression} for the definition of the notation $\rho_\bm$, $A_\bm$, $\delta_\bm$ etc.). 
We remark that a necessary {and sufficient} condition for \eqref{eq:Cond2} to hold is that 
\begin{align}
\label{eq:Cond2Nec}
\frac{\rho_\bm}{\rho_\bmp} = \frac{\rho_\bn}{\rho_\bnp} .
\end{align}
\begin{prop}
\label{lem:SimilaritySufficientEquiv}
Let $Tx=A_Tx+\delta_T$ and $T'x=A_{T'}x+\delta_{T'}$. Then condition \eqref{eq:SimilarityCond} holds if and only if the following three conditions are satisfied:
\begin{align}
\label{eq:Cond2}
\rho_\bm \rho_\bn^{-1} = \rho_\bmp \rho_\bnp^{-1} = \varrho,
\end{align}
\begin{align}
\label{eq:Cond3Equiv}
A_\bm A_T A_\bn^{-1} = A_\bmp A_{T'} A_\bnp^{-1},
\end{align}
\begin{align}
\label{eq:Cond1}
\delta_\bm - \delta_\bmp - \rho_\bm A_\bm (A_T\rho_\bn^{-1} A_\bn^{-1} \delta_\bn-\delta_T) +\rho_\bmp  A_\bmp (A_{T'}\rho_{\bnp}^{-1}A_\bnp^{-1} \delta_\bnp-\delta_{T'}) = 0.
\end{align}
\end{prop}

\begin{proof}
We first note that \eqref{eq:SimilarityCond} is equivalent to
\begin{align}
&|\rho_\bm A_\bm (A_T\rho_\bn^{-1} A_\bn^{-1}(x-\delta_\bn)+\delta_T)+\delta_\bm \notag\\
\label{eq:SimCondEquiv}
&\qquad \qquad \qquad - \rho_\bmp A_\bmp (A_{T'}\rho_\bnp^{-1} A_\bnp^{-1}(y-\delta_\bnp)+{\delta_{T'}}) -\delta_\bmp| = \varrho |x-y|, \quad  x,y\in\R^n. 
\end{align}
It is easy to check that \eqref{eq:Cond2}-\eqref{eq:Cond1} are sufficient for \eqref{eq:SimCondEquiv} (and hence for \eqref{eq:SimilarityCond}). 
To see that they are also necessary, 
suppose that \eqref{eq:SimCondEquiv} holds. Then taking $x=y=0$ in \eqref{eq:SimCondEquiv} gives \eqref{eq:Cond1}. Combining this with \eqref{eq:SimCondEquiv} gives
\begin{align}
\label{eq:SimCondEquiv2}
|\rho_\bm \rho_\bn^{-1}A_\bm A_T A_\bn^{-1}x - \rho_\bmp \rho_\bnp^{-1}A_\bmp A_{T'} A_\bnp^{-1}y| = \varrho |x-y|,\quad  x,y\in\R^n, 
\end{align}
and taking first $x=0$ and $y\neq 0$, then $x\neq 0$ and $y= 0$ in this equation gives
\begin{align*}
|\rho_\bmp \rho_\bnp^{-1}A_\bmp A_{T'}A_\bnp^{-1}y| = \rho_\bmp \rho_\bnp^{-1}|y|=\varrho|y|,& \quad  y\in\R^n,\\
|\rho_\bm \rho_\bn^{-1}A_\bm A_T A_\bn^{-1}x| = \rho_\bm \rho_\bn^{-1}|x|=\varrho|x|,& \quad  x\in\R^n,
\end{align*}
from which we deduce \eqref{eq:Cond2}.
Finally, combining \eqref{eq:Cond2} with \eqref{eq:SimCondEquiv2} gives 
\begin{align*}
|A_\bm A_T A_\bn^{-1}x - A_\bmp A_{T'}A_\bnp^{-1}y| = |x-y|, \quad  x,y\in\R^n, 
\end{align*}
or, equivalently,
\begin{align}
\label{eq:SimCondEquiv4}
|x - (A_\bm A_T A_\bn^{-1})^{-1} A_\bmp A_{T'} A_\bnp^{-1}y| = |x-y|, \quad  x,y\in\R^n. 
\end{align}
Now we note that if $B$ is an $n\times n$ matrix and $|x-By|=|x-y|$ for all $x,y\in\R^n$ then $B$ is the identity matrix. To prove this, suppose that $B$ is not the identity matrix. Then there exists $y\neq0$ such that $By\neq y$, and setting $x=By$ gives $|x-By|=0$ and $|x-y|\neq 0$, so that $|x-By|\neq |x-y|$. 
Hence \eqref{eq:SimCondEquiv4} implies that
\begin{align}
\label{eq:Cond3}
(A_\bm A_T A_\bn^{-1})^{-1}A_\bmp A_{T'} A_\bnp^{-1}
\end{align}
{is the identity matrix}, which is equivalent to \eqref{eq:Cond3Equiv}.
\end{proof}

\begin{rem}
\label{rem:SimilaritySpecialCase}
If $A_{\bm}$, $A_{\bmp}$, $A_{\bn}$ and $A_{\bnp}$  are all equal to the identity matrix, and $A_{T}=A_{T'}$, then condition \eqref{eq:Cond3Equiv} is automatically satisfied. If, further, $\delta_T$ and $\delta_{T'}$ are both zero, and condition \eqref{eq:Cond2} holds, then condition \eqref{eq:Cond1} reduces to 
\begin{align}
\label{eq:Cond1Special}
\delta_\bm - \delta_\bmp = \varrho A_T(\delta_\bn - \delta_\bnp).
\end{align}
\end{rem}

\section{Algorithm for deriving representation formulas}
\label{sec:Algorithm}

We present our algorithm for deriving representation formulas for the integral $I_{\Gamma,\Gamma}$ in Algorithm \ref{alg:Algorithm} below. 
The output of the algorithm, when it terminates, is a linear system of the form
\begin{align}
\label{eq:LinearSystem}
A\mathbf{x} = B\mathbf{r} + \mathbf{b},
\end{align}
where $\mathbf{x}=[I_{\bm^{}_{s,1},\bm'_{s,1}},\ldots,I_{\bm^{}_{s,n_s},\bm'_{s,n_s}}]^T\in \R^{n_s}$ is a vector of ``fundamental'' singular sub-integrals {(the subscript $_s$ standing for ``singular'')}, with the original integral $I_{\bm^{}_{s,1},\bm'_{s,1}}=I_{\Gamma,\Gamma}$ as its first entry,
$\mathbf{r}=[I_{\bm^{}_{r,1},\bm'_{r,1}},\ldots,I_{\bm^{}_{r,n_r},\bm'_{r,n_r}}]^T\in \R^{n_r}$ is a vector of ``fundamental'' regular sub-integrals {(the subscript $_r$ standing for ``regular'')}, $A\in \R^{n_s\times n_s}$ and $B\in \R^{n_s\times n_r}$ are matrices, and $\mathbf{b}\in \R^{n_s}$ is a vector of logarithmic terms present only in the case $t=0$.  
The algorithm is based on repeated subdivision of the integration domain and the identification of similarities between the resulting sub-integrals (in the sense of \eqref{eq:SimilarityDefn}), and  
the word ``fundamental'' refers to a sub-integral which, when encountered in the subdivision algorithm, is not found to be similar to any other sub-integral previously encountered. 
Whether the algorithm terminates, and the resulting lengths $n_s$ and $n_r$ of the vectors $\mathbf{x}$ and $\mathbf{r}$, depends on the measures $\mu$ and $\mu'$, as we shall demonstrate in \S\ref{sec:examples}.

If the algorithm terminates, 
one can obtain numerical approximations to the values of the integrals in $\mathbf{x}$, including the original integral $I_{\Gamma,\Gamma}$, by applying a suitable quadrature rule to the regular integrals in $\mathbf{r}$, solving the system \eqref{eq:LinearSystem}, and extracting the relevant entry from the solution vector $\mathbf{x}$. We discuss this in more detail in \S\ref{sec:Quadrature}.

\begin{rem}
\label{rem:Disjoint}
If $\Gamma$ is disjoint then the algorithm terminates with $n_s=1$ and recovers the result of \cite[Thm~4.6]{HausdorffQUAD}.
\end{rem}

\begin{rem}
\label{rem:Subdivision}
Our algorithm {(in line 9)} requires us to specify a subdivision strategy. With reference to the notation in Algorithm \ref{alg:Algorithm}, we have considered two such strategies:
\begin{itemize}
\item Strategy 1: always subdivide both $\Gamma_{\bm_{s,n}}$ and $\Gamma_{\bm'_{s,n}}$, i.e.\ take (with $({\emptyset},m)$ interpreted as $(m)$)
\begin{align}
\label{eq:Strategy1}
\mathcal{I}_n = \{(\bm_{s,n},m)\}_{m=1}^M, \quad \mathcal{I}'_{n} = \{(\bm'_{s,n},m)\}_{m=1}^M. 
\end{align}
\item Strategy 2: subdivide only the larger of $\Gamma_{\bm_{s,n}}$ and $\Gamma_{\bm'_{s,n}}$, i.e.\ take
\begin{align}
\label{eq:Strategy2}
\begin{cases}
\mathcal{I}_n = \{(\bm_{s,n},m)\}_{m=1}^M \text{ and } \mathcal{I}'_{n} = \{\bm'_{s,n}\}, & \text{if } \diam(\Gamma_{\bm_{s,n}})>\diam(\Gamma_{\bm'_{s,n}}),\\
\mathcal{I}_n = \{\bm_{s,n}\} \text{ and } \mathcal{I}'_{n} = \{(\bm'_{s,n},m)\}_{m=1}^M, & \text{if } \diam(\Gamma_{\bm_{s,n}})<\diam(\Gamma_{\bm'_{s,n}}),\\
\mathcal{I}_n = \{(\bm_{s,n},m)\}_{m=1}^M\text{ and } \mathcal{I}'_{n} = \{(\bm'_{s,n},m)\}_{m=1}^M, & \text{if } \diam(\Gamma_{\bm_{s,n}})=\diam(\Gamma_{\bm'_{s,n}}).
\end{cases}
\end{align}
\end{itemize}
If the IFS is homogeneous then the two subdivision strategies coincide, since using Strategy 2 we never encounter pairs of subsets with different diameters. 
\end{rem}

\begin{rem}
\label{rem:Singular}
Our algorithm {(in line 11)} requires the user to be able to determine whether an integral $I_{\bn,\bnp}$ is singular or regular, i.e.\ whether $\Gamma_\bn$ intersects ${\Gamma_\bnp}$ non-trivially or not. Deriving a criterion for this based solely on the indices $\bn$ and $\bnp$ and the IFS parameters appears to be an open problem. However, for the examples considered in \S\ref{sec:examples} we were able to determine this by inspection on a case-by-case basis. 
{We emphasize that one does not need to specify the type of singularity, i.e.\ the dimension of $\Gamma_\bn\cap {\Gamma_\bnp}$, merely whether $\Gamma_\bn\cap {\Gamma_\bnp}$ is empty or not.}
\end{rem}

\begin{rem}
\label{rem:Similarities}
Our algorithm {(in lines 12 and 20)} requires a way of checking for ``similarity'' of pairs of subintegrals. For this we use Propositions \ref{lem:Fubini}-\ref{lem:SimilaritySufficientEquiv}, {combined with a user-provided list of isometries $T$ and $T'$ under which $\mu$ and $\mu'$ are respectively invariant. Then, 
}
to verify \eqref{eq:SimilarityCond} in Proposition \ref{lem:SimilaritySufficient}, we use Proposition \ref{lem:SimilaritySufficientEquiv}: we first check \eqref{eq:Cond2}, then \eqref{eq:Cond3Equiv}, then \eqref{eq:Cond1}. As noted in Remark \ref{rem:SimilaritySpecialCase}, in the special case where 
$A_{\bm}$, $A_{\bmp}$, $A_{\bn}$ and $A_{\bnp}$ are all equal to the identity matrix, $A_{T}=A_{T'}$, and $\delta_T$ and $\delta_{T'}$ are both zero, it is sufficient to check \eqref{eq:Cond2} and then \eqref{eq:Cond1Special}. 

{
The question of how to determine the permitted isometries $T$ and $T'$ was discussed in Remark \ref{rem:Isometries}. 
If the user is not able to provide the full list of isometries for $\mu$ and $\mu'$, it may be that the algorithm still terminates, but does so with a larger number $n_s$ of fundamental singular integrals than would be obtained with the full list of isometries. 
However, in \S\ref{sec:Interval} we provide an example where failing to specify a non-trivial isometry would lead to non-termination of the algorithm.}

\end{rem}

\begin{rem}
{
The matrix $A$ (when the algorithm terminates) depends not only on the IFS $\{s_1,\ldots,s_M\}$ and the weights $p_1,\ldots,p_M$, but also on the subdivision strategy used in line 9 of the algorithm. For both subdivision strategies described in Remark \ref{rem:Subdivision}, the first column of $A$ is guaranteed to be of the form $(\alpha,0\ldots,0)^T$ for some $\alpha>0$, because of the fact that $\mu(\Gamma_m\cap\Gamma_{m'})=0$ for $m\neq m'$. 
For all the examples in \S\ref{sec:Sierpinski}-\S\ref{sec:Snowflake} 
the matrix $A$ is upper triangular, with diagonal entries that are all non-zero when $t<t_*$, so that $A$ is invertible when $t<t_*$. 
However, upper-triangularity of $A$ is not guaranteed in general, as \S\ref{sec:Interval} illustrates (see in particular \eqref{eq:AInterval}), and proving that $A$ is invertible whenever the algorithm terminates remains an open problem.
}
\end{rem}

\begin{center}
\begin{algorithm2e}[t!]
\begin{center}
\fbox{\parbox{.89\textwidth}{
\SetAlgoLined
\nl Initialise $A=[\,]$, 
$B=[\,]$, 
$\mathbf{r}=[\,]$, 
$\mathbf{b}=[\,]$, 
$n_r=0$,   
$\mathbf{x}=[I_{\Gamma,\Gamma}]$, $n_s=1$, and $n=0$. \\
\nl \While{$n<n_s$}{
\nl Increment $n=n+1$.\\
\nl Append a $1\times n_s$ row of zeros to $A$.\\ \nl Set $A_{n,n}=1$.\\
\nl \If{$B$ \emph{is non-empty}}{\nl Append a $1\times n_r$ row of zeros to $B$.} %
\nl Append a single zero entry to $\mathbf{b}$. \\
\nl Subdivide $\Gamma_{\bm_{s,n}}=\bigcup_{\bn\in\mathcal{I}_n}\Gamma_\bn$ and $\Gamma_{\bm'_{s,n}}=\bigcup_{\bnp\in\mathcal{I}'_n}\Gamma_\bnp$ according to some subdivision strategy producing index sets $\mathcal{I}_n,\mathcal{I}'_n\subset {\Sigma}$. \\
\nl \For{\emph{each pair} $(\bn,\bnp)\in \mathcal{I}_n\times \mathcal{I}'_n$}{
\nl \eIf{$I_{\bn,\bnp}$ \emph{is singular}}{
\nl \eIf{$I_{\bn,\bnp}$ \emph{is similar to an existing entry of} $\mathbf{x}$}{
\nl Let $i(\bn,\bnp)$, $a(\bn,\bnp)$ and $b(\bn,\bnp)$ be such that \\ \mbox{\qquad\quad}$I_{\bn,\bnp}=a(\bn,\bnp)I_{\bm_{s,i(\bn,\bnp)},\bm_{s,i(\bn,\bnp)}'}+b(\bn,\bnp)$. \\
\nl Update $A_{n,i(\bn,\bnp)}=A_{n,i(\bn,\bnp)}-a(\bn,\bnp)$. \\
\nl Update $b_{n}=b_{n}+b(\bn,\bnp)$.
}
{\nl Add $I_{\bn,\bnp}$ as a new entry in $\mathbf{x}$. \\
\nl Increment $n_s=n_s+1$. \\
\nl Augment $A$ by an $n\times 1$ column of zeros. \\
\nl Set $A_{n,n_s}=-1$.
}
}
{\nl \eIf{$I_{\bn,\bnp}$ \emph{is similar to an existing entry of} $\mathbf{r}$}
{\nl Let $i(\bn,\bnp)$ and $a(\bn,\bnp)$ and $b(\bn,\bnp)$ be such that \\\mbox{\qquad\quad}$I_{\bn,\bnp}=a(\bn,\bnp)I_{\bm_{r,i(\bn,\bnp)},\bm_{r,i(\bn,\bnp)}'}+b(\bn,\bnp)$. \\
\nl Update $B_{n,i(\bn,\bnp)}=B_{n,i(\bn,\bnp)}+a(\bn,\bnp)$. \\
\nl Update $b_{n}=b_{n}+b(\bn,\bnp)$.
}
{\nl Add $I_{\bn,\bnp}$ as a new entry in $\mathbf{r}$. \\\nl Increment $n_r=n_r+1$.\\
\nl Augment $B$ by an $n\times 1$ column of zeros. \\
\nl Set $B_{n,n_r}=1$.
}
}
}
}
}}
\caption{Algorithm for deriving representation formulas for the integral \eqref{eq:integral}.}
\label{alg:Algorithm}
\end{center}
\end{algorithm2e}
\end{center}

\clearpage
\section{Examples}
\label{sec:examples}
We now report the results of applying Algorithm \ref{alg:Algorithm} to some standard examples. 

\subsection{Sierpinski triangle, Hausdorff measure}
\label{sec:Sierpinski}

\begin{figure}[t!]
\centering
\includegraphics[width=\textwidth]{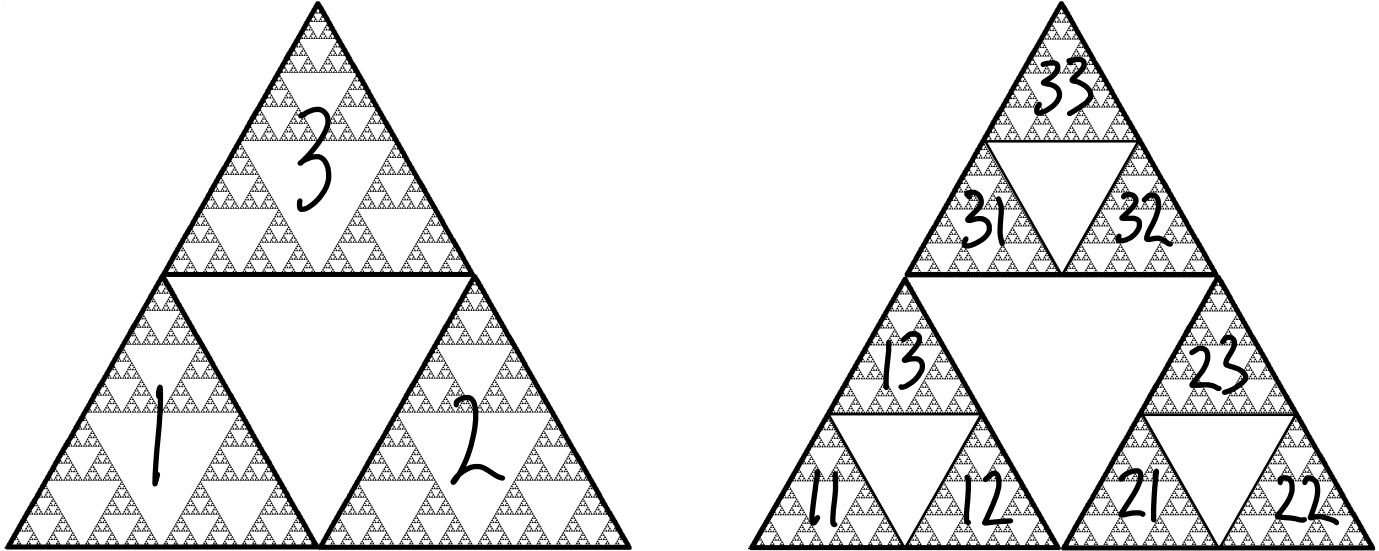}
\caption{Level 1 and 2 subsets of the Sierpinski triangle.}
\label{fig:SierpinskiTriangle}
\end{figure}

We first consider the case where $\Gamma\subset\R^2$ is the Sierpinski triangle, the attractor of the homogeneous IFS with $M=3$ and 
\[ s_1(x)=\frac12x, \quad s_2(x)=\frac12 x+\left(\frac12,0\right), \quad s_3(x)=\frac12x+\left(\frac14,\frac{\sqrt3}{4}\right),
\]
for which $d=\dimH(\Gamma)=\log{3}/\log2\approx 1.59$. The first two levels of subdivision of $\Gamma$ are illustrated in Figure \ref{fig:SierpinskiTriangle}. 
We shall assume for simplicity that $\mu=\mu'=\cH^d|_\Gamma$, %
so that $t_*=d$. 
{Since we are working with Hausdorff measure, as discussed in Remark \ref{rem:Isometries} the isometries $T$ under which $\mu$ is invariant are precisely those for which $T(\Gamma)=\Gamma$, which in this case are the elements of the dihedral group $D_3$ corresponding to the symmetries of the equilateral triangle.} 
The two subdivision strategies \eqref{eq:Strategy1} and \eqref{eq:Strategy2} coincide, since the IFS is homogeneous, and Algorithm \ref{alg:Algorithm} terminates after finding two fundamental singular sub-integrals, with $\bx=(I_{\Gamma,\Gamma},I_{1,2})^T$. %
The integral $I_{1,2}$ captures the interaction between neighbouring subsets of $\Gamma$ of the same size, intersecting at a point. 
The linear system \eqref{eq:LinearSystem} satisfied by these unknowns is:
\begin{align}
\label{}
\left(\begin{array}{cc}
\sigma_1 & -6 \\
0 & \sigma_2 
\end{array}\right)
\left(\begin{array}{c}
I_{\Gamma,\Gamma}\\ 
I_{1,2}
\end{array}\right)
= 
\left(\begin{array}{c}
0\\ 
R_{1,2}
\end{array}\right)
+ \left(\begin{array}{c}
b_1\\ 
b_2
\end{array}\right),
\end{align}
where 
\[\sigma_1 = 1-\frac{2^{t}}{3}\in\Big(0,\frac{2}{3}\Big],\quad \sigma_2 = 1-\frac{2^{t}}{9}\in\Big(\frac{2}{3},\frac{8}{9}\Big], \qquad t\in[0,d),
\]
\begin{align}
\label{}
\left(\begin{array}{c}
b_1\\ 
b_2
\end{array}\right)
=
\begin{cases}
(0,0)^T, & t\in(0,d),\\[3mm]
\displaystyle -\frac{\log 2\cH^d(\Gamma)^2}{3}\left(1, 
\frac{1}{27}
\right)^T, & t=0,
\end{cases}
\end{align}
and 
\[ R_{1,2}=3I_{11,21}+I_{11,22} + 2I_{11,23} + 2I_{12,23}\]
is the sum of the regular integrals arising from the decomposition of $I_{1,2}$ into level 2 subsets. In the notation of \S\ref{sec:Algorithm} we have
\[
A = \left(\begin{array}{cc}
\sigma_1 & -6 \\
0 & \sigma_2 
\end{array}\right),
\qquad
B=\left(\begin{array}{cccc}
0 & 0 & 0 & 0\\
3 & 1 & 2 & 2
\end{array}\right),
\]
and 
$\mathbf{r} = (I_{11,21}, I_{11,22}, I_{11,23}, I_{12,23})^T$. Then
\[
A^{-1} = \left(\begin{array}{cc}
\dfrac{1}{\sigma_1} & \dfrac{6}{\sigma_1\sigma_2} \\[2mm]
0 & \dfrac{1}{\sigma_2} 
\end{array}\right),
\]
and, solving the system, we obtain the representation formulas
\begin{align}
I_{1,2} = 
\frac{1}{\sigma_2}(R_{1,2}+b_2),
\qquad 
\label{eq:IGammaGamma_Sierpinski_final}
I_{\Gamma,\Gamma} = 
\frac{1}{\sigma_1}\left(\frac{6}{\sigma_2}(R_{1,2}+b_2)+b_1\right).
\end{align}

\subsection{Vicsek fractal, Hausdorff measure}

\begin{figure}[t!]
	\centering
	\includegraphics[width=0.9\textwidth]{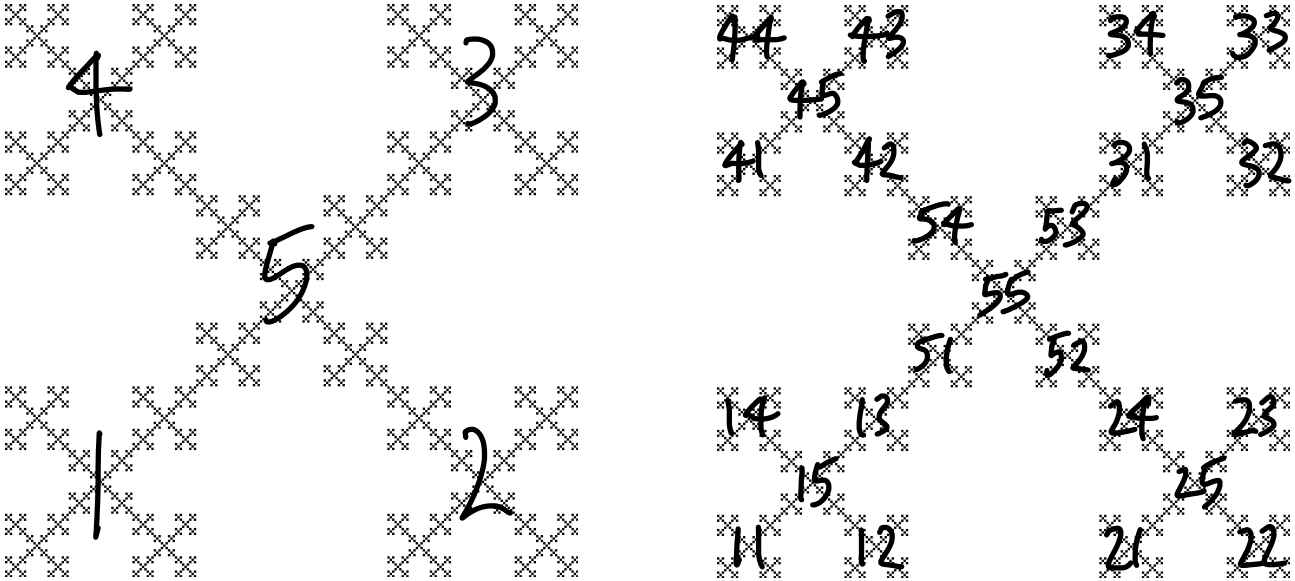}
	\caption{Level 1 and 2 subsets of the Vicsek fractal.}
	\label{fig:Vicsek}
\end{figure}

Next, we consider the case where $\Gamma\subset\R^2$ is the Vicsek fractal (shown in Figure \ref{fig:Vicsek}), the attractor of the homogeneous IFS with $M=5$ and
\begin{align*}
s_1(x)=&\frac13x, \quad s_2(x)=\frac13 x+\left(\frac23,0\right), \quad s_3(x)=\frac13x+\left(\frac23,\frac23\right),
\\ &s_4(x)=\frac13x+\left(0,\frac23\right),
\quad s_5(x)=\frac13x+\left(\frac13,\frac13\right),
\end{align*}
with $d=\dimH(\Gamma)=\log5/\log3\approx1.47$. Again we assume that $\mu=\mu'=\cH^d|_\Gamma$, %
so that $t_*=d$. 
{In this case the isometries $T$ under which $\mu$ is invariant are the elements of the dihedral group $D_4$ corresponding to the symmetries of the square.} 
The first two levels of subdivision of $\Gamma$ are illustrated in Figure \ref{fig:Vicsek}, from which it is clear that the situation is similar to that for the Sierpinski triangle of \S\ref{sec:Sierpinski}, as the only new singularities at level one are point singularities, which are similar (in the sense of \eqref{eq:SimilarityCond}) to those arising at level two. Again, our two subdivision strategies coincide because the IFS is homogeneous, and Algorithm \ref{alg:Algorithm} terminates after finding two fundamental singular sub-integrals, with $\bx=\left(I_{\Gamma,\Gamma},I_{1,5}\right)^T$.
For brevity we do not present the full linear system satisfied by these unknowns, but rather just report the matrix $A$ of \eqref{eq:LinearSystem}, which is
\[
A = \left(\begin{array}{cc}
\sigma_1 & -8 \\
0 & \sigma_2 
\end{array}\right),
\qquad \text{so that} \qquad
A^{-1} = \left(\begin{array}{cc}
\dfrac{1}{\sigma_1} & \dfrac{8}{\sigma_1\sigma_2} \\[2mm]
0 & \dfrac{1}{\sigma_2}
\end{array}\right)
,
\]
where
\[\sigma_1 = 1-\frac{3^t}{5}\in\Big(0,\frac{4}{5}\Big],\quad\sigma_2=1-\frac{3^t}{25}\in\Big(\frac{4}{5},\frac{24}{25}\Big],\quad t\in[0,d).
\]

\subsection{Sierpinski carpet, Hausdorff measure}
\label{sec:Carpet}

\begin{figure}[t!]
\centering
\includegraphics[width=\textwidth]{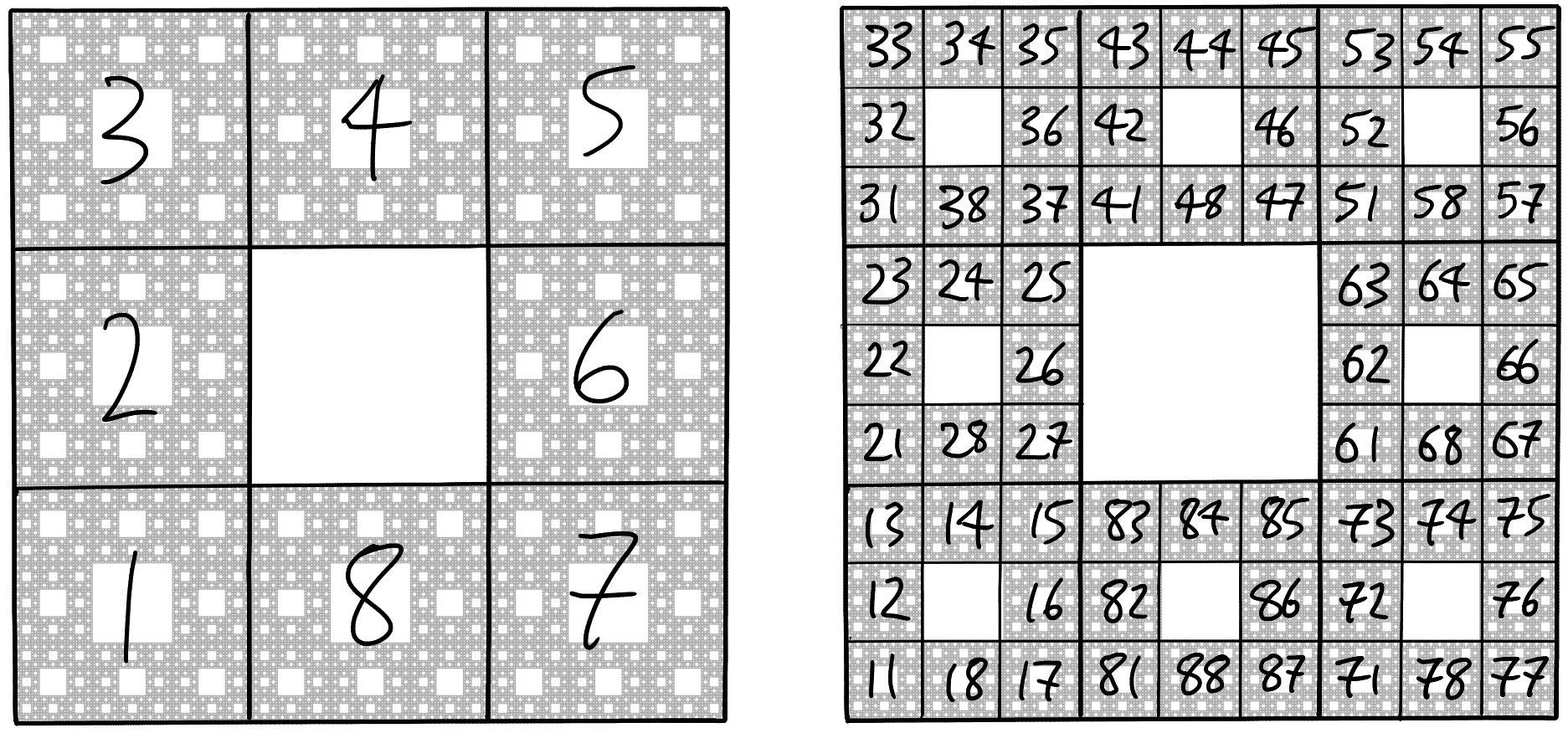}
\caption{Level 1 and 2 subsets of the Sierpinski carpet.}
\label{fig:SierpinskiCarpet}
\end{figure}

Next, we consider the case where $\Gamma\subset\R^2$ is the Sierpinski carpet, the attractor of the homogeneous IFS with $M=8$ and 
\[ s_1(x)=\frac13x, \quad s_2(x)=\frac13 x+\left(0,\frac13\right), \quad s_3(x)=\frac13x+\left(0,\frac23\right),\quad s_4(x)=\frac13x+\left(\frac13,\frac23\right),
\]
\[ s_5(x)=\frac13x+\left(\frac23,\frac23\right), \quad s_6(x)=\frac13 x+\left(\frac23,\frac13\right), \quad s_7(x)=\frac13x+\left(\frac23,0\right),\quad s_8(x)=\frac13x+\left(\frac13,0\right),
\]
for which $d=\dimH(\Gamma)=\log{8}/\log3\approx 1.89$. The first two levels of subdivision of $\Gamma$ are illustrated in Figure \ref{fig:SierpinskiTriangle}. 
We again assume that $\mu=\mu'=\cH^d|_\Gamma$, %
so that $t_*=d$. 
{As for the previous example, the isometries $T$ under which $\mu$ is invariant are the elements of the dihedral group $D_4$.} 
Again the two subdivision strategies \eqref{eq:Strategy1} and \eqref{eq:Strategy2} coincide, since the IFS is homogeneous, and now Algorithm \ref{alg:Algorithm} terminates after finding three fundamental singular sub-integrals, with $\mathbf{x} = (I_{\Gamma,\Gamma}, I_{1,2}, I_{2,4})^T$. The integrals $I_{1,2}$ and $I_{2,4}$ capture the interaction between neighbouring subsets of the same size, intersecting along a line segment and at a point, respectively. 
In this case, the matrix of \eqref{eq:LinearSystem} is:
\begin{align*}
\label{}
A = \left(\begin{array}{ccc}
\sigma_1 & -16 & -8 \\
0 & \sigma_2 & -\sigma_4 \\
0 & 0 & \sigma_3 
\end{array}\right),
\qquad \text{so that} \qquad
A^{-1}=\left(
\begin{array}{ccc}
	\dfrac{1}{\sigma_1} & \dfrac{16}{\sigma_1\sigma_2} & \dfrac{8\sigma_2+16\sigma_4}{\sigma_1\sigma_2\sigma_3} \\[2mm]
	0 & \dfrac{1}{\sigma_2} & \dfrac{\sigma_4}{\sigma_2\sigma_3} \\[2mm]
0 & 0 & \dfrac{1}{\sigma_3} \\
\end{array}
\right)
,
\end{align*}
where, for $t\in [0,d)$,
\[\sigma_1 = 1-\frac{3^t}{8}\in\Big(0,\frac{7}{8}\Big] , 
\quad \sigma_2 = 1-\frac{3.3^t}{64}\in\Big(\frac{5}{8},\frac{61}{64}\Big],
\quad \sigma_3 = 1-\frac{3^t}{64}\in\Big(\frac{7}{8},\frac{63}{64}\Big],
\quad \sigma_4 = \frac{3^t}{16}\in\Big(\frac{1}{16},\frac{1}{2}\Big].
\]

\subsection{Koch snowflake, Lebesgue measure}
\label{sec:Snowflake}

\begin{figure}[t!]
\centering
\includegraphics[width=\textwidth]{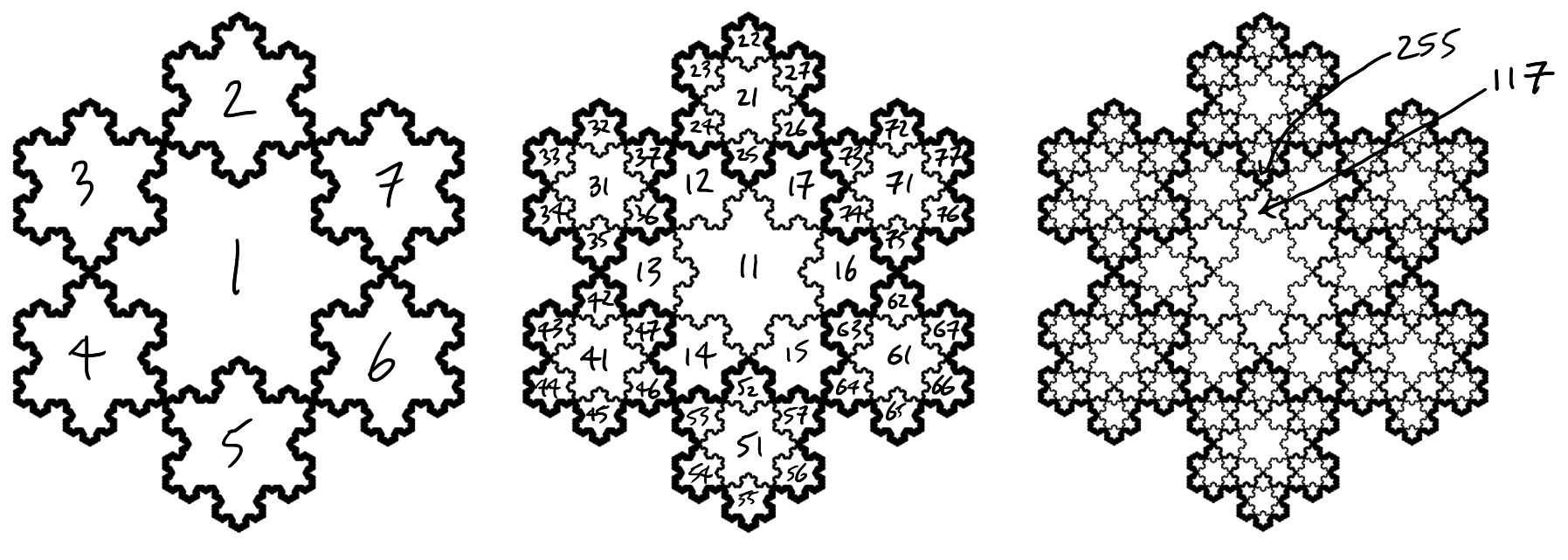}
\caption{The level one, two and three subdivisions of the Koch snowflake.}
\label{fig:KochSnowflake}
\end{figure}

Next, we consider the case where $\Gamma\subset\R^2$ is the Koch snowflake, the attractor of the non-homogeneous IFS with $M=7$ and
\[s_1(x)=\frac{1}{\sqrt{3}}A_1x, 
\quad A_1 = 
\left(\begin{array}{cc}
\frac{\sqrt{3}}{2} & -\frac{1}{2} \\
\frac{1}{2} & \frac{\sqrt{3}}{2}
\end{array}\right)
\quad \text{(anticlockwise rotation by }\pi/6\text{)},
\]
\[s_m(x)=\frac13x+\frac23(\cos\alpha_{m},\sin\alpha_{m}), \quad \alpha_m=\frac{(2m-1)\pi}6,\quad m=2,\ldots,7,\]
for which $d=\dimH(\Gamma)=2$. The first three levels of subdivision of $\Gamma$ are illustrated in Figure \ref{fig:KochSnowflake}. We assume that $\mu$ and $\mu'$ are both equal to the Lebesgue measure on $\R^2$, restricted to $\Gamma$,  %
so, again, $t_*=d$. 
{(As mentioned in Example \ref{ex:Hausdorff}, $\mu$ is proportional to $\cH^2|_\Gamma$.) 
The isometries $T$ under which $\mu$ is invariant are the elements of the dihedral group $D_6$ corresponding to the symmetries of the hexagon.} 
In this case both subdivision strategies produce a terminating algorithm, but with different results.

With Strategy 1 (subdividing both subsets) Algorithm \ref{alg:Algorithm} terminates after finding four fundamental singular sub-integrals, with $\mathbf{x} = (I_{\Gamma,\Gamma}, I_{1,2}, I_{2,3}, I_{11,25})^T$. 
The integral $I_{1,2}$ captures the interaction between neighbouring subsets of $\Gamma$, intersecting along a Koch curve. The integrals $I_{2,3}$ and $I_{11,25}$ both capture point interactions, but of different types: in $I_{2,3}$ the two interacting subsets are the same size, while in $I_{11,25}$ one is three times the diameter of the other. $I_{11,25}$ arises as a new fundamental sub-integral in the subdivision of $I_{1,2}$, but in the subdivision of $I_{11,25}$ one obtains just one singular sub-integral, $I_{117,255}$, which is similar to $I_{11,25}$, so the algorithm terminates. 

For brevity we do not report the resulting linear system satisfied by $(I_{\Gamma,\Gamma}, I_{1,2}, I_{2,3}, I_{11,25})^T$, but instead present the simpler result obtained with Strategy 2 (subdividing the subset with the largest diameter), for which 
Algorithm \ref{alg:Algorithm} terminates after finding only three fundamental singular sub-integrals, with $\mathbf{x} = (I_{\Gamma,\Gamma}, I_{1,2}, I_{2,3})^T$. With this strategy, in the subdivision of $I_{1,2}$ we subdivide only $\Gamma_1$, leaving $\Gamma_2$ intact, obtaining the edge interaction sub-integrals $I_{12,2}$ and $I_{17,2}$, both of which are similar to $I_{1,2}$, and the point interaction sub-integral $I_{11,2}$, which is similar to $I_{2,3}$.
The resulting linear system is
\begin{align}
\label{}
\left(\begin{array}{ccc}
\sigma_1 & -12 & -12 \\
0 & \sigma_2 & -1 \\
0 & 0 & \sigma_3 
\end{array}\right)
\left(\begin{array}{c}
I_{\Gamma,\Gamma}\\ 
I_{1,2}\\
I_{2,3}
\end{array}\right)
= 
\left(\begin{array}{c}
R_{\Gamma,\Gamma}\\ 
R_{1,2}\\
R_{2,3}
\end{array}\right)
+
\left(\begin{array}{c}
b_1\\ 
b_2\\
b_3
\end{array}\right),
\end{align}
where 
\[\sigma_1 = 1-\frac{3^{t/2}}{9} - \frac{2.3^t}{27}\in\Big(0,\frac{22}{27}\Big], 
\quad \sigma_2 = 1-\frac{2.3^{t/2}}{9}\in \Big(\frac{1}{3},\frac{7}{9}\Big],
\quad \sigma_3 = 1-\frac{3^t}{81}\in \Big(\frac{26}{27},\frac{80}{81}\Big],
\qquad t\in [0,d),\]
\begin{align}
\label{}
\left(\begin{array}{c}
b_1\\ 
b_2\\
b_3
\end{array}\right)
=
\begin{cases}
(0,0,0)^T, & t\in(0,d),\\[3mm]
\displaystyle -\frac{\log 3|\Gamma|^2}{27}\left(\frac{7}{2}, 
\frac{1}{9},
\frac{41}{13122}
\right)^T, & t=0,
\end{cases}
\end{align}
and
\[R_{\Gamma,\Gamma} = 12I_{2,4} + 6I_{2,5}, \qquad R_{1,2}=2I_{13,2} + 2I_{14,2},\]
\begin{align*}
\label{}
R_{2,3}%
=&\left(\frac{3^{t/2}}{9}+\frac{4.3^t}{81}\right)I_{2,4}
+\frac{2.3^t}{81}I_{2,5}+\frac{2.3^{t/2}}{9}I_{13,2}+\frac{4.3^{t/2}}{9}I_{14,2}+4I_{21,33}+2I_{21,34}
+6I_{22,32}
\\
& +2I_{22,33}+4I_{22,34}+2I_{22,35}+4I_{22,36}+6I_{22,37}
+4I_{23,34}+I_{27,34}
\end{align*}
are linear combinations of regular integrals.  %

Hence 
\[ A = \left(\begin{array}{ccc}
\sigma_1 & -12 & -12 \\
0 & \sigma_2 & -1 \\
0 & 0 & \sigma_3 
\end{array}\right),
\qquad \text{so that} \qquad
A^{-1} = \left(\begin{array}{ccc}
\dfrac{1}{\sigma_1} & \dfrac{12}{\sigma_1\sigma_2} & \dfrac{12(1+\sigma_2)}{\sigma_1\sigma_2\sigma_3} \\[2mm]
0 & \dfrac{1}{\sigma_2} & \dfrac{1}{\sigma_2\sigma_3} \\[2mm]
0 & 0 & \dfrac{1}{\sigma_3}
\end{array}\right),
\]
and solving the system gives
\[ \quad I_{2,3} = \frac{1}{\sigma_3}\Big(R_{2,3} + b_3\Big),
\quad I_{1,2} = \frac{1}{\sigma_2}\left(R_{1,2}+b_2+\frac{1}{\sigma_3}\Big(R_{2,3} + b_3\Big)\right),
\] 
and 
\begin{align}
\label{eq:IGammaGamma_Koch_final}
I_{\Gamma,\Gamma} = \frac{1}{\sigma_1}\left(R_{\Gamma,\Gamma} +b_1+\frac{12}{\sigma_2} \Big(R_{1,2}+b_2\Big)+\frac{12}{\sigma_3}\left(1+\frac{1}{\sigma_2}\right)\Big(R_{2,3} + b_3\Big)\right).
\end{align}

\subsection{$\Gamma=[0,1]$, including non-terminating examples}
\label{sec:Interval}
\begin{figure}[t!]
\centering
\includegraphics[width=.5\textwidth]{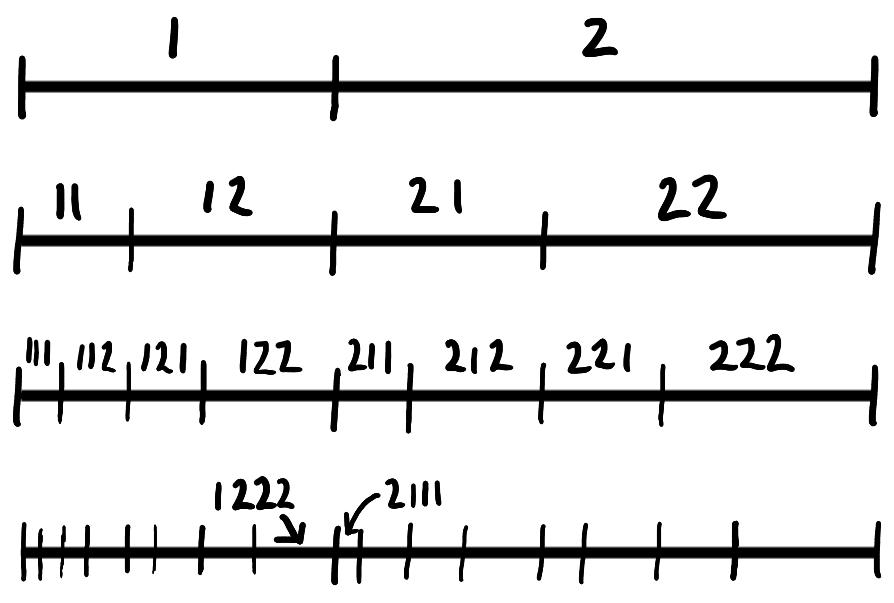}
\caption{The level one, two, three and four subdivisions of the interval $[0,1]$, viewed as the attractor of the IFS $\{s_1,s_2\}$ with $s_1(x)=\rho x$ and $s_2(x)=(1-\rho)x+\rho$ for some $\rho\in(0,1/2)$.}
\label{fig:Interval}
\end{figure}

We now consider 
a class of simple one-dimensional examples that illustrates the dependence of the output of Algorithm \ref{alg:Algorithm} on the choice of measures $\mu$ and $\mu'$, and the fact that it does not always terminate.

Given $\rho\in(0,1/2]$, consider the IFS $\{s_1,s_2\}$ in $\R$ with $s_1(x)=\rho x$ and $s_2(x)=(1-\rho)x+\rho$, for which $\Gamma=[0,1]$. 
The first four levels of subdivision of $\Gamma$ are illustrated in Figure \ref{fig:Interval} in the case $\rho\in(0,1/2)$. 
Let $\mu$ be the invariant measure on $\Gamma$ for some weights $p_1,p_2\in(0,1)$ with $p_1+p_2=1$, and let $\mu'=\mu$, so that $I_{\bm,\bmp}=I_{\bmp,\bm}$ for all $\bm,\bmp\in {\Sigma}$ by Proposition \ref{lem:Fubini}. For definiteness we assume the normalisation $\mu(\Gamma)=1$. 
If $p_1=\rho$ then $p_2=1-\rho$ and $\mu$ is Lebesgue measure restricted to $[0,1]$ {(recall Example \ref{ex:Hausdorff})}, which is invariant under the operation {$T_{\rm ref}$} of reflection with respect to the point $x=1/2$. If $p_1\neq\rho$ then $\mu$ is not Lebesgue measure and 
{the only isometry $T$ under which $\mu$ is invariant is the identity}.

If $\rho=1/2$ then the IFS is homogeneous, so  %
Strategy 1 and Strategy 2 coincide, and Algorithm \ref{alg:Algorithm} always terminates, for any %
$\mu$, finding 
just 2 fundamental singular sub-integrals, $I_{\Gamma,\Gamma}$ and $I_{1,2}$.

If $\rho\in(0,1/2)$ then the IFS is inhomogeneous, so Strategy 1 and Strategy 2 differ, and the outcome of Algorithm \ref{alg:Algorithm} depends on the choice of strategy, and on the measure $\mu$. We consider four cases:
\begin{itemize}
\item \underline{Case 1: Lebesgue measure, Strategy 1}
 
Algorithm \ref{alg:Algorithm} terminates, finding 3 fundamental singular sub-integrals, $I_{\Gamma,\Gamma}$, $I_{1,2}$ and $I_{12,21}$.\footnote{{The matrix $A$ arising in \eqref{eq:LinearSystem} in this case, with $\mathbf{x} = (I_{\Gamma,\Gamma}, I_{1,2}, I_{12,21})^T$,  is given by
\begin{align}
\label{eq:AInterval}
A = \left(\begin{array}{ccc}
1-\omega_1-\omega_2 & -2 & 0 \\
0 & 1 & -1 \\
0 & -\omega_1\omega_2 & 1
\end{array}\right),
\end{align}
where $\omega_1 = \rho^{2-t}$ and $\omega_2 = (1-\rho)^{2-t}$. Noting that $\det(A)=(1-\omega_1-\omega_2)(1-\omega_1\omega_2)$, one can check that $A$ is invertible for all $(\rho,t)\in (0,1/2]\times[0,2)$.
}}
($I_{122,211}$ is similar to $I_{2,1}=I_{1,2}$ in this case, {taking $T=T'=T_{\rm ref}$. This is an example where, if the non-trivial isometry $T_{\rm ref}$ had not been identified, the algorithm would not have terminated - cf.\ Case 3 below.)} 
\item \underline{Case 2: Lebesgue measure, Strategy 2}

If $\rho=\rho_*:=\frac{3-\sqrt{5}}{2}\approx 0.38$, the unique positive solution of $\rho=(1-\rho)^2$, Algorithm \ref{alg:Algorithm} terminates with two fundamental singular sub-integrals $I_{\Gamma,\Gamma}$ and $I_{1,2}$. (In this case $I_{1,21}$ is similar to $I_{2,1}=I_{1,2}$, {again taking $T=T'=T_{\rm ref}$}.) 

If $\rho\in(0,\rho_*)\cup(\rho_*,1/2)$ then Algorithm \ref{alg:Algorithm} terminates with four fundamental singular sub-integrals $I_{\Gamma,\Gamma}$, $I_{1,2}$, $I_{1,21}$ and $I_{12,21}$. (The subdivision of $I_{12,21}$ leads to $I_{122,211}$, which 
is similar to $I_{2,1}=I_{1,2}$, as in Case 1.)
\item \underline{Case 3: Non-Lebesgue measure, Strategy 1}:

In this case Algorithm \ref{alg:Algorithm} does not terminate, since we encounter an infinite sequence of fundamental singular sub-integrals 
\begin{align}
\label{eq:SeqNoTerminate}
I_{\Gamma,\Gamma}, I_{1,2},  I_{12,21}, I_{122,211}, I_{1222,2111},I_{12222,21111} \ldots,
\end{align}
none of which is found to be similar to any other. %
To see this, note that for a sub-integral $\Gamma_{\bm,\bmp}$ in this sequence, with $\bm=(1,2,\ldots ,2)$ ($k$ 2's) and $\bmp=(2,1,\ldots ,1)$ ($k$ 1's), we have  
\begin{align}
\label{eq:SeqRatios}
\frac{\rho_\bm}{\rho_\bmp}=\frac{\rho(1-\rho)^k}{\rho^k(1-\rho)} = \left(\frac{1-\rho}{\rho}\right)^{k-1}=:R_{k}.
\end{align}
Then since $0<\rho<1-\rho<1$, the sequence $(R_k)_{k=0}^\infty$ is monotonically increasing, with $R_k\geq 1$ for $k\geq 1$. This implies that \eqref{eq:Cond2Nec} (and hence \eqref{eq:Cond2}) is not satisfied by any pair of elements of the sequence \eqref{eq:SeqNoTerminate}, {except for $I_{1,2}$ and $I_{122,211}$. However, since $\mu$ is not Lebesgue measure, $I_{1,2}$ and $I_{122,211}$ are not found to be similar, because $\mu$ is not invariant under $T_{\rm ref}$ (so one cannot use it in Proposition \ref{lem:SimilaritySufficient}), and \eqref{eq:Cond1} fails with $T=T'$ the identity. 
}

\item \underline{Case 4: Non-Lebesgue measure, Strategy 2}: 

In this case, Algorithm \ref{alg:Algorithm} can only terminate if $\rho$ is a solution of a polynomial equation
\begin{align}
\label{eq:RhoCond}
(1-\rho)^{j}\rho^k=1, 
\quad \text{ or } \quad 
(1-\rho)^{j}=\rho^k,
\end{align}
for some $j,k\in\N_0$ with either $j>0$ or $k>0$. In particular, Algorithm \ref{alg:Algorithm} does not terminate if $\rho$ is transcendental. 
To see that \eqref{eq:RhoCond} is necessary for termination of the algorithm, we note that, in the subdivision of $I_{1,2}$, Strategy 2 will produce pairs of subsets of $\Gamma$ (intervals) that intersect at the point $x=\rho$, and the lengths of the intervals in each pair will be in the ratio $\rho(1-\rho)^j:\rho^k(1-\rho)$ for some $j,k\in\N_0$. For the sub-integrals associated to any two distinct pairs of such intervals to be found to be similar, the ratio of their lengths must coincide (by \eqref{eq:Cond2Nec}), implying that 
\[\frac{\rho(1-\rho)^{j_1}}{\rho^{k_1}(1-\rho)}=\frac{\rho(1-\rho)^{j_2}}{\rho^{k_2}(1-\rho)}, \qquad \text{or, equivalently, }\,\,\,\,(1-\rho)^{j_1-j_2}\rho^{{k_2-k_1}}=1,\]
for some $j_1,k_1,j_2,k_2\in\N_0$ with either $j_1\neq j_2$ or $k_1\neq k_2$, giving \eqref{eq:RhoCond}. 

If $\rho$ is the solution of a polynomial \eqref{eq:RhoCond} then Algorithm \ref{alg:Algorithm} may terminate, but the number of fundamental singular sub-integrals encountered will depend on $\rho$. 
For instance, if $\rho=\rho_*$ the algorithm terminates with four fundamental singular sub-integrals $I_{\Gamma,\Gamma}$, $I_{1,2}$, $I_{1,21}$ and $I_{12,21}$ (since in this case $I_{122,211}$ is similar to $I_{1,21}$). 
If $\rho=\rho_{**}\approx 0.43$, defined to be the unique positive solution of $\rho^2=(1-\rho)^3$, or $\rho=\rho_{***}\approx 0.32$, defined to be the unique positive solution of $\rho=(1-\rho)^3$, the algorithm terminates with five fundamental singular sub-integrals $I_{\Gamma,\Gamma}$, $I_{1,2}$, $I_{1,21}$, $I_{12,21}$ and $I_{122,211}$ (since in these cases $I_{1222,211}$ is similar to $I_{1,2}$ and $I_{1,21}$ respectively). For $\rho\in(0,1/2)\setminus\{\rho_*,\rho_{**},\rho_{***}\}$, if the algorithm does terminate it will find at least six fundamental singular sub-integrals, since then none of $I_{\Gamma,\Gamma}$, $I_{1,2}$, $I_{1,21}$, $I_{12,21}$, $I_{122,211}$ and $I_{1222,211}$ are found to be similar to each other. 

\end{itemize}

\section{Numerical quadrature and error estimates}
\label{sec:Quadrature} 

Once Algorithm \ref{alg:Algorithm} has been applied, and the system \eqref{eq:LinearSystem} has been solved, producing a representation formula for the singular integral $I_{\Gamma,\Gamma}$ in terms of regular sub-integrals, a numerical approximation of $I_{\Gamma,\Gamma}$ can be obtained by applying a suitable quadrature rule to the regular sub-integrals. Let us call the resulting approximation $Q_{\Gamma,\Gamma}$. We discuss some possible choices of quadrature rule below. But first we make a general comment on the error analysis of such approximations. 
Suppose that the quadrature rule chosen can compute each of the regular sub-integrals in the vector $\mathbf{r}$ with absolute error $\leq E$ for some $E\geq0$. Then the absolute quadrature error in computing $I_{\Gamma,\Gamma}$ using the representation formula \eqref{eq:LinearSystem} can be bounded by
\begin{align}
\label{eq:QuadError}
|I_{\Gamma,\Gamma}-Q_{\Gamma,\Gamma}| \leq \|A^{-1}\|_\infty\|B\|_\infty E.
\end{align}
The constant $\|A^{-1}\|_\infty\|B\|_\infty$ depends on the problem at hand, and is expected to blow up as $t\to t_*$. Indeed, for the examples in \S\ref{sec:Sierpinski}-\S\ref{sec:Snowflake} (for which $t_*=d=\dimH(\Gamma)$) one can check %
that 
\[ \|A^{-1}\|_\infty\|B\|_\infty \leq \frac{C}{d-t}, \qquad t\in[0,d),\]
for some constant $C>0$, independent of $t$. This follows from the fact that in all these examples the constant $\sigma_1=O(d-t)$ as $t\to d$, while $\sigma_2$, $\sigma_3$ etc.\ remain bounded away from zero in this limit.

We now return to the choice of quadrature rule for the approximation of the regular sub-integrals in $\mathbf{r}$, which are double integrals of smooth functions over pairs of self-similar subsets of $\Gamma$ with respect to a pair of invariant measures $\mu$ and $\mu'$. We shall restrict our attention to tensor product quadrature rules, so that it suffices to consider methods for evaluating a single integral of a smooth function over a single self-similar subset $\Gamma_\bm$ of $\Gamma$ with respect to a single invariant measure $\mu$. In fact, it is enough to consider the case $\Gamma_\bm = \Gamma$, since the more general case can then be treated using \eqref{eq:cov_general}. Hence we consider quadrature rules for the evaluation of the integral
\begin{equation}\label{eq:singleint}
	I[f;\mu]:=\int_\Gamma f~\rd\mu,
	\end{equation}
for an integrand $f$ that is smooth in a neighbourhood of $\Gamma$. 
We consider three types of quadrature:
\begin{itemize}
\item \textbf{Gauss rules:} Highly accurate, but currently only practically applicable for the case $n=1$, i.e.\ $\Gamma\subset \R$. \footnote{Clenshaw-Curtis rules have also been studied in this context (see e.g.\ \cite{CaCo:09}), but since Gauss and Clenshaw-Curtis rules typically converge at a similar rate (see e.g.\ \cite{Tr:08} for the classical case), we shall for brevity restrict our attention to the discussion of Gauss rules here.}
\item \textbf{Composite barycentre rules:} Less accurate than Gauss rules, but can be applied to $\Gamma\subset\R^n$ for $n>1$. 
\item \textbf{Chaos game rules:} Monte-Carlo type rules which converge (in expectation) at a relatively slow but dimension-independent rate, which makes them well-suited to high-dimensional problems (large $d=\dimH(\Gamma)$). 
\end{itemize}

In the following three sections we provide further details of these methods, and any theory supporting them, before comparing their performance numerically in \S\ref{sec:Numerics}.

\subsection{Gauss rules in the case $\Gamma\subset\R$}\label{sec:gauss}

In general, $N$-point Gauss rules require the existence of a set of polynomials $\{p_j\}_{j=0}^N$, orthogonal with respect to the measure $\mu$. A sufficient condition for the existence of such polynomials is positivity of the Hankel determinant, which in the case of self-similar invariant measures, is implied by $\supp\mu=\Gamma$ having infinitely many points \cite[\S1.1]{Ga:90}.
We then define the $N$-point Gauss rule on $\Gamma$ as
\begin{equation}\label{eq:Gaussdef}
Q^{\mathrm{G}}_N[f;\mu] := \sum_{j=1}^N w_jf(x_j),
\end{equation}
where $x_j$, $j=1,\ldots,N$, are the zeros of $p_N$. Gauss rules are interpolatory, so the weights (also known as Christoffel numbers) may be defined by $w_j:=\int_\Gamma \ell_j(x)\rd\mu(x)$, where $\ell_j$ is the $j$th corresponding Lagrange polynomial (see e.g \cite[(5.3)]{Tr:13}). The weights are positive - see for example \cite[Theorem~1.46]{Ga:04}, which generalises to any positive measure $\mu$.

For classical $\mu$, a range of algorithms (see e.g. \cite{HaTo:13,ToTrOl:16}) exist for efficient $O(N)$ computation of the weights and nodes in \eqref{eq:Gaussdef}. However, standard approaches involving polynomial sampling break down for singular measures \cite{Ma:96,Ma:00}. This presents an obstacle for the evaluation of \eqref{eq:singleint} in our context of self-similar invariant measures, which are in general singular when $\dimH(\Gamma)\neq n$. %
However, in the special case $n=1$, where $\Gamma\subset\R$, this issue can be overcome by applying the stable Stieltjes technique proposed in \cite[\S5]{Ma:96}\footnote{There is an error in \cite[Equation (32)]{Ma:96}: %
	\begin{align}
	\label{}
	\sum_{m=0}^{n-1}\Gamma^{n}_{i,m}\Gamma^n_{i,m+1}\delta_i(r_m+r_{m+1})\quad
	\text{ should be replaced by }\quad2\sum_{m=0}^{n-1}\delta_i\Gamma^n_{i,m}\Gamma^n_{i,m+1}r_{m+1}. 	
	\end{align}
}. 
It seems that a stable and efficient algorithm for the evaluation of Gauss rules for the case where $\Gamma\subset \R^n$, $n>1$, has not yet been developed.
Hence in this paper we only consider Gauss rules for the case where $\Gamma\subset\R$.

The error analysis for the Gauss rule follows the standard approach, giving the usual exponential convergence as $N\to\infty$. 
{In the following, $\Hull(\Gamma)$ denotes the convex hull of $\Gamma$.}
\begin{thm}
\label{thm:Gauss}
Let $\Gamma\subset \R$ be an IFS attractor and let $\mu$ be a self-similar measure supported on $\Gamma$. If $f$ is analytic in a neighbourhood of $\Hull(\Gamma)\subset\R$, then
	\[
	\left|I[f;\mu]-Q^{\mathrm{G}}_N[f;\mu]\right| \leq C\re^{-cN}, \qquad N\in\N,
	\]
for some constants $C>0$ and $c>0$, independent of $N$.
\end{thm}
\begin{proof}
	Denote by $p_{2N-1}^*$ the $L^\infty(\Hull(\Gamma))$-best approximation to $f$, over the space of polynomials of degree $2N-1$. By linearity and the exactness property \cite[(1.17)]{Ga:90}, we can write
	\begin{align*}
	\left|I[f;\mu]-Q^{\mathrm{G}}_N[f;\mu]\right| &= \left|I[f-p_{2N-1}^*;\mu]-Q^{\mathrm{G}}_N[f-p_{2N-1}^*;\mu]\right| \\
	&\leq \left|I[f-p_{2N-1}^*;\mu]\right|+\left|Q^{\mathrm{G}}_N[f-p_{2N-1}^*;\mu]\right|\\
	&\leq \|f-p_{2N-1}^*\|_{L^\infty(\Hull(\Gamma))}\Big(\mu(\Gamma) + \sum_{j=1}^N|w_j|\Big).
	\end{align*}
	Since the weights are positive, $ \sum_{j=1}^N|w_j|= \sum_{j=1}^Nw_j = \mu(\Gamma)$. The result then follows by applying classical approximation theory estimates to $\|f-p_{2N-1}^*\|_{L^\infty(\Hull(\Gamma))}$, for example \cite[Theorem~8.2]{Tr:13}.
\end{proof}

\subsection{The composite barycentre rule}\label{sec:bary}
The basic idea of the composite barycentre rule (for more detail see \cite{HausdorffQUAD}) is to partition $\Gamma$ into a union of self-similar subsets of approximately equal diameter, then to approximate $f$ on each subset by its (constant) value at the barycentre of each subset.  %
Given a maximum mesh width $h>0$, we define a partition of $\Gamma$ using the 
following set of indices:
\begin{align}
\label{eq:Lh_def}
L_{h}(\Gamma) := \big\{&\bm=(m_1,\ldots,m_\ell)\in {\Sigma} :
 \diam(\Gamma_{\bm})\leq h \text{ and } \diam(\Gamma_{(m_1,\ldots,m_{\ell-1})})>h\big\}.\end{align}
The composite barycentre rule is then defined as
\begin{equation}\label{eq:bary}
Q_h^{\mathrm{B}}[f;\mu] := \sum_{\bm\in L_h(\Gamma)} w_\bm f(x_\bm),
\end{equation}
where the weights and nodes are defined by $w_\bm := \mu(\Gamma_\bm)$ and $x_\bm:={\int_{\Gamma_\bm} x~\rd\mu(x)}/{\mu(\Gamma_\bm)}$ respectively. The weights and nodes can be computed using simple formulas involving the IFS parameters of \eqref{eq:affine}, as (see \cite[(28)-(30)]{HausdorffQUAD}, and recall \eqref{eq:sim_measure_general})
\begin{align}
\label{eq:weightsnodes}
w_{\bm}= p_\bm\mu(\Gamma), \qquad
x_{\bm} = s_{\bm}(x_\Gamma), %
\end{align}
with 
\begin{align}
\label{eq:xGamma_formula}
x_\Gamma:=\frac{\int_{\Gamma}x\,\rd\mu(x)}{\int_{\Gamma}\rd \mu(x)}=\Bigg(\mathrm{I} - \sum_{m=1}^M p_m\rho_mA_{m}\Bigg)^{-1}\bigg(\sum_{m=1}^M p_m\delta_m\bigg),
\end{align}
where $\mathrm{I}$ is the $n\times n$ identity matrix and $\rho_m$, $A_m$ and $\delta_m$, $m=1,\ldots,M$, are as in \eqref{eq:affine}.

The error analysis of the composite barycentre rule follows a standard Taylor series approximation argument. The following is a simplified version of results in \cite{HausdorffQUAD}. 
\begin{thm}[{\cite[Theorem~3.6 and Remark~3.9]{HausdorffQUAD}}]
\label{thm:Barycentre}
Let $\Gamma\subset \R^n$ be an IFS attractor and let $\mu$ be a self-similar measure supported on $\Gamma$. 
\begin{enumerate}[(i)]
\item If $f$ is Lipschitz continuous on $\Hull(\Gamma)$ then
\[
\left|I[f;\mu]-Q^{\mathrm{B}}_h[f;\mu]\right| \leq Ch, \qquad h>0,
\]
for some $C>0$ independent of $h$. 
\item 
If $f$ is differentiable in a neighbourhood of $\Hull(\Gamma)$, and its gradient is Lipschitz continuous on $\Hull(\Gamma)$ then 
\[
\left|I[f;\mu]-Q^{\mathrm{B}}_h[f;\mu]\right| \leq Ch^2, \qquad h>0,
\]
for some $C>0$ independent of $h$.
\end{enumerate}
If the IFS defining $\Gamma$ is homogeneous then $h$ and $h^2$ on the right-hand sides of the above estimates can be replaced by $N^{-1/d}$ and $N^{-2/d}$ respectively, where $N:=|L_h(\Gamma)|$.
\end{thm}

\subsection{Chaos game quadrature}\label{sec:chaos}

Chaos game quadrature, described, e.g., in \cite[(3.22)--(3.23)]{forte1998chaos} and \cite[\S~6.3.1]{kunze2011fractal}, is a Monte-Carlo type approach, defined by the following procedure:
\begin{enumerate}[(i)]
	\item Choose some $x_0\in\mathbb{R}^n$, e.g.\ $x_0=x_\Gamma$, the barycentre of $\Gamma$; %
	\item Select a realisation of the sequence $\{m_j\}_{j\in\mathbb N}$ of i.i.d.\ random variables taking values in $\{1,\ldots,M\}$ with probabilities $\{p_1,\ldots,p_M\}$;
	\item Construct the stochastic sequence $x_j=s_{m_j}(x_{j-1})$ for $j\in\mathbb N$;
	\item 
For a given $N\in\N$, define the chaos game quadrature approximation by
\begin{equation}\label{eq:Chaos}
Q^{\mathrm{C}}_N[f;\mu] := \frac{1}{N}\sum_{j=1}^N f(x_j).
\end{equation}
\end{enumerate}
For continuous $f$, the chaos game rule \eqref{eq:Chaos} will converge to \eqref{eq:singleint} with probability one (see the arguments in the appendix of \cite{forte1998chaos}).
While no error estimates were provided in \cite{forte1998chaos} or \cite{kunze2011fractal}, in the numerical experiments of \cite[\S6]{HausdorffQUAD} and \S\ref{sec:Numerics} below, convergence in expectation was observed at a rate consistent with an estimate of the form
 \[
 \mathbb{E}\left[\left|I[f;\mu]-Q^{\mathrm{C}}_N[f;\mu]\right|\right]\leq CN^{-1/2}, \qquad N\in\N.
 \]

\section{Numerical results and applications}
\label{sec:Numerics} 
Algorithm \ref{alg:Algorithm}, and the quadrature approximations described in \S\ref{sec:Quadrature}, have been implemented in the open-source Julia code \verb|IFSIntegrals|, available at \url{www.github.com/AndrewGibbs/IFSintegrals}. In this section we present numerical results illustrating the accuracy of our approximations, comparing different quadrature approaches, and applying our method in the context of a boundary element method for acoustic scattering by fractal screens.

\subsection{Sierpinski triangle, Vicsek fractal, Sierpinski carpet, Koch snowflake}\label{sec:nvo}

\begin{table}[t!]
	\centering
\resizebox{\textwidth}{!}{
\begin{tabular}{|c|c||c|c|c|c|c|c|c|c||c|c|c|}
\hline 
& & $p_1$ & $p_2$ & $p_3$ & $p_4$ & $p_5$ & $p_6$ & $p_7$ & $p_8$ & $t_*$ & $n_s$ & $n_r$
\\
\hline
Sierpinski & $\mu$ & 0.3631 & 0.4921 & 0.1448 & - & - & - & - & - & \multirow{2}{*}{1.3303} & \multirow{2}{*}{7} & \multirow{2}{*}{30}\\
triangle & $\mu'$ & 0.6520 & 0.3183 & 0.0297 & - & - & - & - & - & & & \\
\hline
Vicsek & $\mu$ & 0.0721 & 0.2664 & 0.3158 & 0.1990 & 0.1467 & - & - & - & \multirow{2}{*}{1.4559} & \multirow{2}{*}{5} & \multirow{2}{*}{52}\\
fractal & $\mu'$ & 0.0942 & 0.1064 & 0.1655 & 0.4130 & 0.2209 & - & - & - & & & \\
\hline
Sierpinski & $\mu$ & 0.2041 & 0.1256 & 0.1605 & 0.0908 & 0.2835 & 0.0083 & 0.0032 & 0.1240 & \multirow{2}{*}{1.6670}& \multirow{2}{*}{9} & \multirow{2}{*}{112}\\
carpet & $\mu'$ & 0.0522 & 0.1507 & 0.2695 & 0.2408 & 0.1951 & 0.0054 & 0.0047 & 0.0815 & & &  \\
\hline
Koch & $\mu$ & 0.0591 & 0.0852 & 0.0621 & 0.2714 & 0.0436 & 0.1867 & 0.2918 &- & \multirow{2}{*}{3.0359} & \multirow{2}{*}{43} & \multirow{2}{*}{468}\\
snowflake & $\mu'$ & 0.1575 & 0.1594 & 0.1182 & 0.1728 & 0.1101 & 0.1482 & 0.1338 &- & & & \\
\hline
\end{tabular}}
\caption{Random weights used for the self-similar measures considered in \S\ref{sec:nvo}, along with the resulting value of $t_*$ solving \eqref{eqn:tstar}, and the numbers $n_s$ and $n_r$ of fundamental singular and regular integrals encountered in Algorithm \ref{alg:Algorithm}.}\label{tab:weights}
\end{table}

\begin{figure}[t!]
	\centering
	\includegraphics[width=0.49\linewidth]{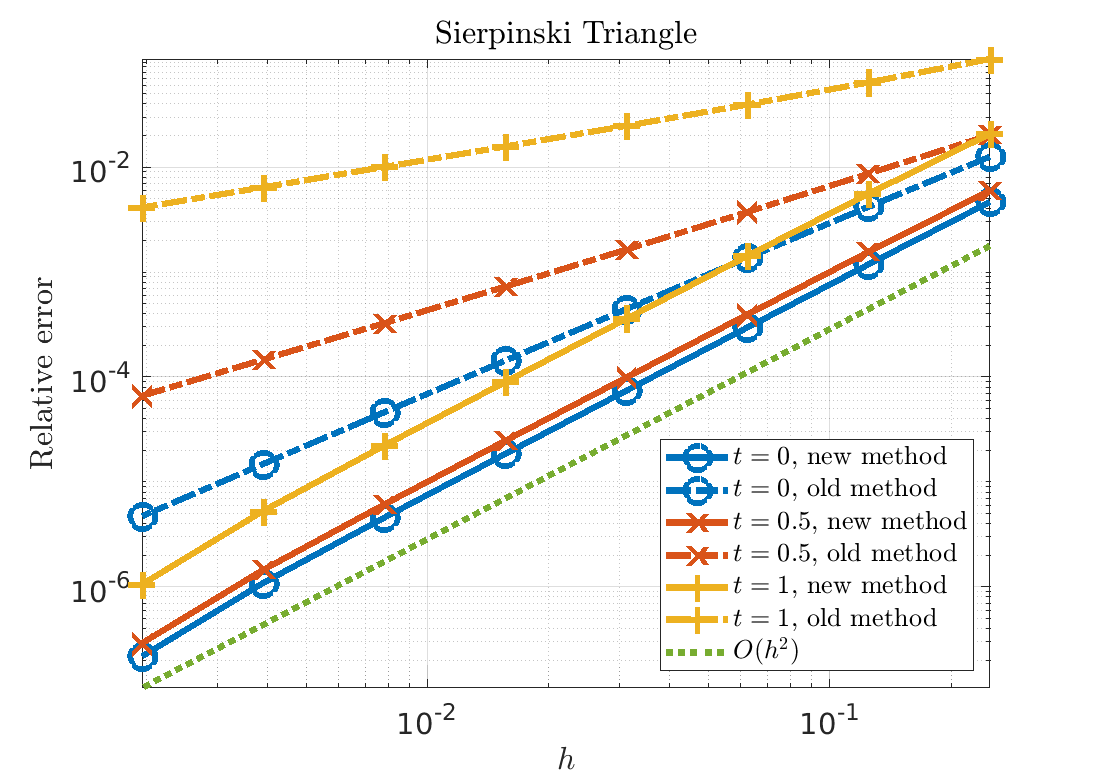}
	\includegraphics[width=0.49\linewidth]{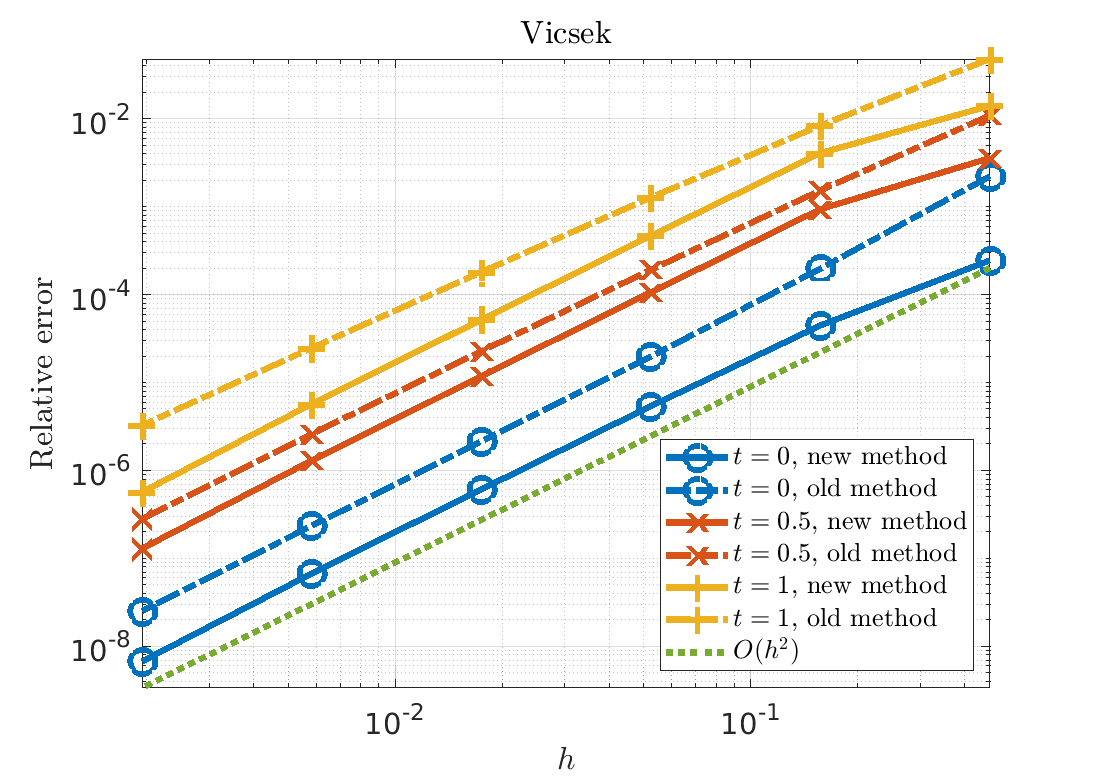}
	\includegraphics[width=0.49\linewidth]{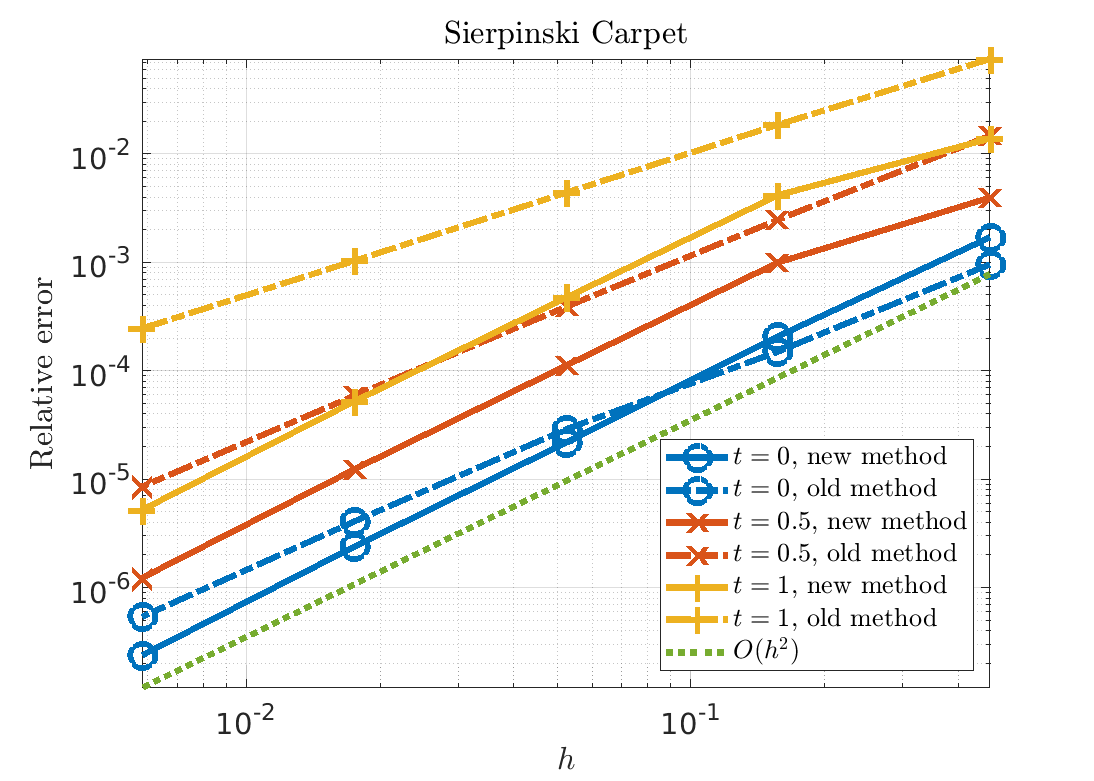}
	\includegraphics[width=0.49\linewidth]{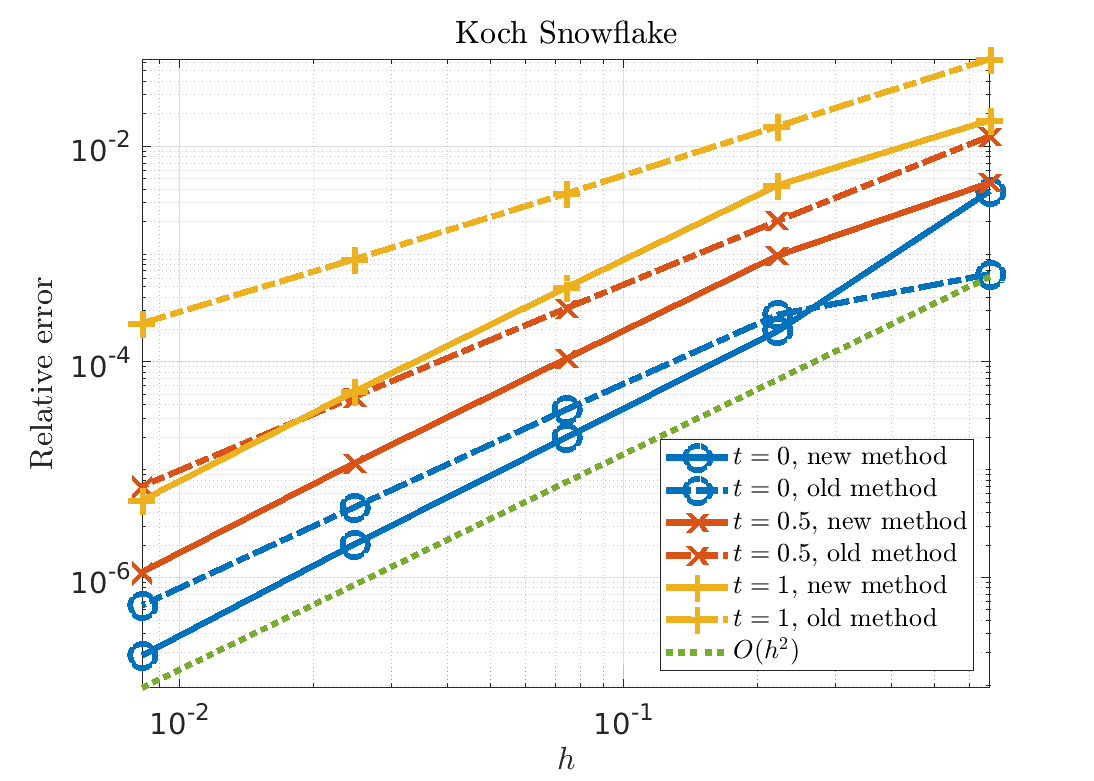}
	\caption{Convergence of composite barycentre rule quadrature approximations for the four examples of \S\ref{sec:nvo}, using the full linear system \eqref{eq:LinearSystem} (labelled ``new method'') and using only the first row of \eqref{eq:LinearSystem} (labelled ``old method'').}
	\label{fig:new-vs-old}
\end{figure}

\begin{table}[t!]
	\centering
\begin{tabular}{|c|c|c|c|c|c|c|c|c|}
\hline
&$t=0$&$t=1/2$&$t=1$\\\hline
Sierpinski triangle& 1.6513 & 1.1475 & 0.6413 \\
\hline
Vicsek fractal& 2.0282 & 2.0232 & 1.8435 \\
\hline
Sierpinski carpet& 1.8384 & 1.7530 & 1.3067 \\
\hline
Koch snowflake& 1.9033 & 1.7340 & 1.2428 \\
\hline
\end{tabular}
\caption{Empirical convergence rates for the old method.}\label{tab:rates}
\end{table}

We first consider the application of our approach to the attractors considered in \S\ref{sec:Sierpinski}-\S\ref{sec:Snowflake}, namely the Sierpinski triangle, Vicsek fractal, Sierpinski carpet and Koch snowflake. However, in contrast to \S\ref{sec:Sierpinski}-\S\ref{sec:Snowflake}, where representation formulas were presented for the standard case where $\mu=\mu'=\cH^d|_\Gamma$ (Lebesgue measure in the case of the Koch snowflake), to demonstrate the generality of our approach we present numerical results for completely generic self-similar measures $\mu\neq\mu'$, with randomly chosen probability weights $p_m$ and $p_m'$, as detailed in Table \ref{tab:weights}. Table \ref{tab:weights} also documents the resulting values of $t_*$, as computed by solving \eqref{eqn:tstar}, as well as the numbers $n_s$, $n_r$, of fundamental singular and regular sub-integrals discovered by our algorithm. 
{In all cases our algorithm terminated, using the same subdivision strategies as in \S\ref{sec:Sierpinski}-\S\ref{sec:Snowflake}, producing an invertible matrix $A$. However, for these non-standard examples there are no nontrivial isometries under which the measures are invariant, so our algorithm took $T=T'$ to be the identity throughout. As a result (cf.\ the related discussion in Remark \ref{rem:Similarities}),  
the linear systems \eqref{eq:LinearSystem} are larger than those obtained in the standard case documented in \S\ref{sec:Sierpinski}-\S\ref{sec:Snowflake}, where additional symmetries of the measures could be exploited.}

In Figure \ref{fig:new-vs-old} (solid curves) we plot the relative error in our quadrature approximation for $I_{\Gamma,\Gamma}$, for three values of $t=0,0.5,1$, obtained by solving the linear system \eqref{eq:LinearSystem} obtained by Algorithm \ref{alg:Algorithm} (using subdivision strategy 2 for the Koch snowflake), combined with composite barycentre rule quadrature for the evaluation of the regular sub-integrals, for different values of the maximum mesh width $h$. In more detail, we plot errors for $h=\diam(\Gamma)\rho^{\ell}$, for $\ell=0,\ldots,\ell_{\mathrm{ref}}-1$, where $\ell_{\mathrm{ref}}$ is the value of $\ell$ used for the reference solution (which is computed using the same method). For the three homogeneous attractors, we take $\rho=\rho_1=\ldots=\rho_M$ (the common contraction factor), while for the Koch snowflake, we take $\rho=1/\sqrt{3}$ (the largest contraction factor).
For the 
Sierpinski triangle $\ell_{\mathrm{ref}}=10$, for the Vicsek fractal $\ell_{\mathrm{ref}}=7$, and for the Sierpinski Carpet and Koch snowflake $\ell_{\mathrm{ref}}=6$.
According to our theory, we expect our method to give $O(h^2)$ error, by Theorem \ref{thm:Barycentre}(ii) and \eqref{eq:QuadError}, and this is exactly the rate we observe in our numerical results in Figure \ref{fig:new-vs-old}. 

In Figure \ref{fig:new-vs-old} (dashed curves) we also show results obtained using the method of our previous paper \cite[(48)]{HausdorffQUAD}, which we refer to as the ``old method''. This method is accurate for disjoint IFS attractors, but is expected to perform less well for non-disjoint attractors, because it only applies self-similarity to deal with the self-interaction integrals, and treats all other sub-integrals as being regular. 
Precisely, the old method corresponds to taking the equation corresponding to the first row in the linear system \eqref{eq:LinearSystem} obtained by Algorithm \ref{alg:Algorithm}, solving this equation for $I_{\Gamma,\Gamma}$, then applying the composite barycentre rule not just to the regular sub-integrals coming from the right-hand side of \eqref{eq:LinearSystem}, but also to the fundamental singular sub-integrals $I_{\bm^{}_{s,i},\bm'_{s,i}}$, $i=2,\ldots,n_s$.
We expect that the resulting quadrature approximation should converge to $I_{\Gamma,\Gamma}$ as $h\to 0$, but at a slower rate than our new method, because of the inaccurate treatment of the singular sub-integrals. This is borne out in our numerical results in Figure \ref{fig:new-vs-old}, with the errors for the old method being significantly larger than those for the new method. To quantify these observations, we present in Table \ref{tab:rates} the empirical convergence rates (computed from the errors for the two smallest $h$ values) observed for the old method for each of the three $t$ values considered. The deviation from $O(h^2)$ convergence is different for each example, but clearly increases as $t$, the strength of the singularity, increases, as one would expect.

For all the experiments in Figure \ref{fig:new-vs-old}, the total number of quadrature points $N_{\rm tot}$ grows like $N_{\rm tot}\approx Ch^{-2d}$ as $h\to 0$, with the value of $C$ depending on the number of fundamental regular sub-integrals that need to be evaluated. (Recall from \S\ref{sec:bary} that for each regular sub-integral we use a tensor product rule with $N^2$ points, where $N\approx C'h^{-d}$ for some $C'$.) 
For each choice of attractor, the value of $C$ for the new method is slightly smaller than that for the old method, because the new method takes greater advantage of similarities between regular sub-integrals. 
The value of $N_{\rm tot}$ used for the reference solutions 
is 
1,291,401,630 for the Sierpinski triangle, 
2,382,812,500 for the Vicsek fractal, 
18,790,481,920 for the Sierpinski carpet, 
and 
379,046,894,100 for the Koch snowflake.

\subsection{Unit interval experiments}\label{sec:intervalexps}

We now consider an attractor $\Gamma\subset\R$, so that we can investigate the performance of the Gauss quadrature discussed in \S\ref{sec:gauss}. The classic example of an IFS attractor $\Gamma\subset\R$ is the Cantor set, but since this is disjoint (in the sense of \eqref{eq:Rdef}), it can already be treated by our old method (of \cite{HausdorffQUAD}). To demonstrate the efficacy of our new method for dealing with non-disjoint attractors we consider the case where $\Gamma=[0,1]\subset\R$, which, as discussed in \S\ref{sec:Interval} (taking $\rho=1/2$), is the attractor of the homogeneous IFS with $M=2$,  $s_1(x)=x/2$ and $s_2(x)=x/2+1/2$. 
We consider the case where $\mu=\mu'$, with $p_1=p_1'=1/3$ and $p_2=p_2'=2/3$, so that, by \eqref{eqn:tstar1}, $t_* = \log(9/5)/\log2\approx 0.848$. As discussed in \S\ref{sec:Interval}, for this problem Algorithm \ref{alg:Algorithm} finds just two fundamental singular sub-integrals, $I_{\Gamma,\Gamma}$ and $I_{1,2}$, and two fundamental regular sub-integrals, $I_{11,21}$ and $I_{11,22}$. 

In Figure \ref{fig:cantor-exps} we report relative errors for the computation of $I_{\Gamma,\Gamma}$ with $t=0$ and $t=1/2$, using Algorithm \ref{alg:Algorithm} combined with Gauss, composite barycentre, and chaos game quadrature for the evaluation of the regular sub-integrals. For each method the total number of quadrature points satisfies $N_{\rm tot}=2N^2$, where $N$ is the number of points used for each of the two iterated integrals in each of the two fundamental regular sub-integrals (recall that we are using tensor product rules). For the composite barycentre rule we have $N\approx \frac{1}{4h}$ in this case.  
As the reference solution we use the result obtained using the Gauss rule with $N=100$, which corresponds to $N_{\rm tot}=20,000$. For the Gauss rule we see the expected root-exponential $O(\re^{-c\sqrt{N_{\rm tot}}})$ convergence predicted by Theorem \ref{thm:Gauss} and \eqref{eq:QuadError}, with $c\approx 1.77$ and $c\approx 2.31$ (for $t=0$ and $t=1/2$ respectively) in this case. As a result, the singular double integral $I_{\Gamma,\Gamma}$ can be computed to machine precision using $N\approx 10$, which corresponds to $N_{\rm tot}\approx 200$ quadrature points.  
The barycentre rule is significantly less accurate, converging like $O(h^2)=O(N^{-2})=O(N_{\rm tot}^{-1})$, as predicted by Theorem \ref{thm:Barycentre} and \eqref{eq:QuadError}. For the chaos game quadrature we computed 1000 realisations, plotting both the errors for each realisation and the average error over all the realisations, which is a proxy for the expected error. The latter is observed to converge like $O(N^{-1/2}) = O(N_{\rm tot}^{-1/4})$, in accordance with the remarks at the end of \S\ref{sec:chaos}.

\begin{figure}[t!]
	\centering
	\includegraphics[width=0.49\linewidth]{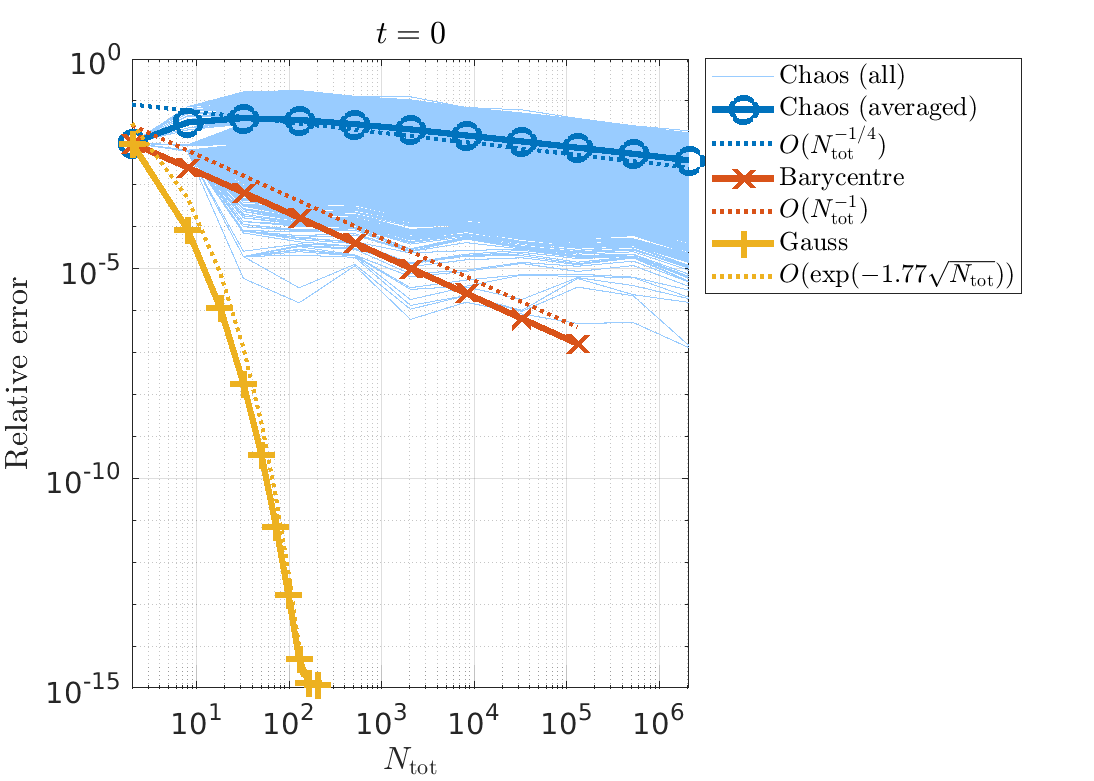}
	\includegraphics[width=0.49\linewidth]{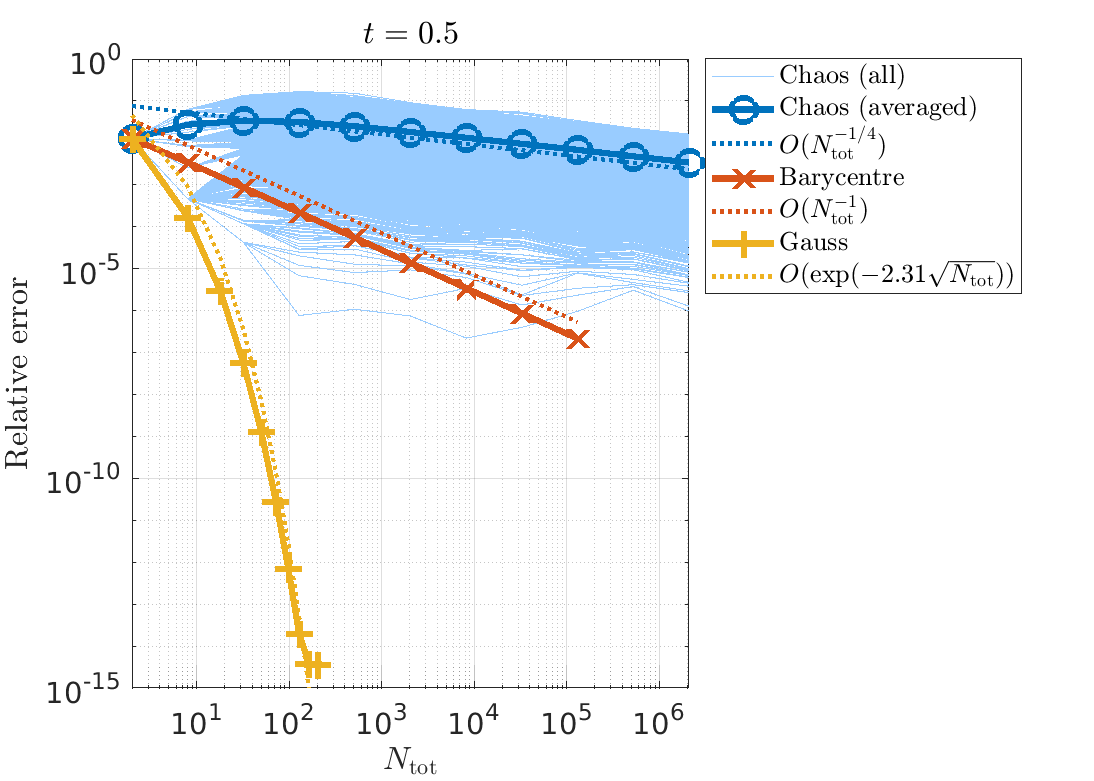}
	\caption{Comparison of the performance of the three quadrature rules presented in \S\ref{sec:gauss}-\S\ref{sec:chaos}, applied in the context of Algorithm \ref{alg:Algorithm}, for the example considered in \S\ref{sec:intervalexps}, which concerns the evaluation of singular double integrals on the unit interval $\Gamma=[0,1]$ with respect to a self-similar measure. 
}
\label{fig:cantor-exps}
\end{figure}

\subsection{Application to Hausdorff BEM for acoustic scattering by fractal screens}
\label{sec:HausdorffBEM}

We conclude by demonstrating how our new representation formulas and resulting quadrature rules can be used to compute the scattering of acoustic waves by fractal screens using the ``Hausdorff boundary element method'' (BEM) of \cite{HausdorffBEM}. For full details of the scattering problem and the Hausdorff BEM we refer the reader to \cite{HausdorffBEM} and the references therein; here we merely provide a brief overview. 

The underlying scattering problem under consideration is the three-dimensional time-harmonic acoustic scattering of an incident plane wave $\re^{\ri kx\cdot \vartheta}$ (for $x=(x_1,x_2,x_3)^T\in \R^3$, wavenumber $k>0$ and unit direction vector $\vartheta\in\R^3$) by a fractal planar screen $\Gamma\times\{0\}\subset\R^3$, where $\Gamma\subset\R^2$ is the attractor of an IFS satisfying the open set condition. Assuming that the total wave field $u$ (which is a solution of the Helmholtz equation $\Delta u + k^2 u=0$ in $\R^3\setminus (\Gamma\times\{0\})$) satisfies homogeneous Dirichlet (``sound soft'') boundary conditions on the screen, it was shown in \cite{HausdorffBEM} that the scattering problem can be reduced to the solution of the integral equation
\begin{align}
\label{eq:BIE}
S \phi = f.
\end{align}
Here $S$ is the single layer boundary integral operator defined by $S\phi(x) = \int_\Gamma \Phi(x,y)\phi(y) \,\rd y$ (with the integral interpreted in a suitable distributional sense), where $\Phi(x,y):=\re^{\ri k |x-y|}/(4\pi|x-y|)$ is the fundamental solution of the Helmholtz equation in three dimensions, $\phi$ is the unknown jump in the $x_3$-derivative of $u$ across the screen, and $f$ is a known function depending on the incident wave. 

The Hausdorff BEM in \cite{HausdorffBEM} discretises \eqref{eq:BIE} using a Galerkin method with a numerical approximation space of piecewise constant functions multiplied by the Hausdorff measure $\cH^d|_\Gamma$. The mesh used for the piecewise constant functions is of the same form as that used in the composite barycentre rule in \S\ref{sec:bary} - having chosen a maximum BEM mesh width $h_{\rm BEM}$ we partition $\Gamma$ using the index set $L_{h_{\rm BEM}}(\Gamma)$ defined in \eqref{eq:Lh_def}. If we choose the natural basis for the approximation space, then assembling the Galerkin matrix involves the numerical evaluation of the integral
\begin{equation}\label{eq:GalInt}
\int_{\Gamma_\bm}\int_{\Gamma_\bn}\Phi(x,y)~\rd\cH^d(y)\rd\cH^d(x), 
\end{equation}
for all pairs of indices $\bm,\bn\in L_{h_{\rm BEM}}(\Gamma)$.

\begin{figure}[t!]
	\centering
	\includegraphics[width=0.49\linewidth]{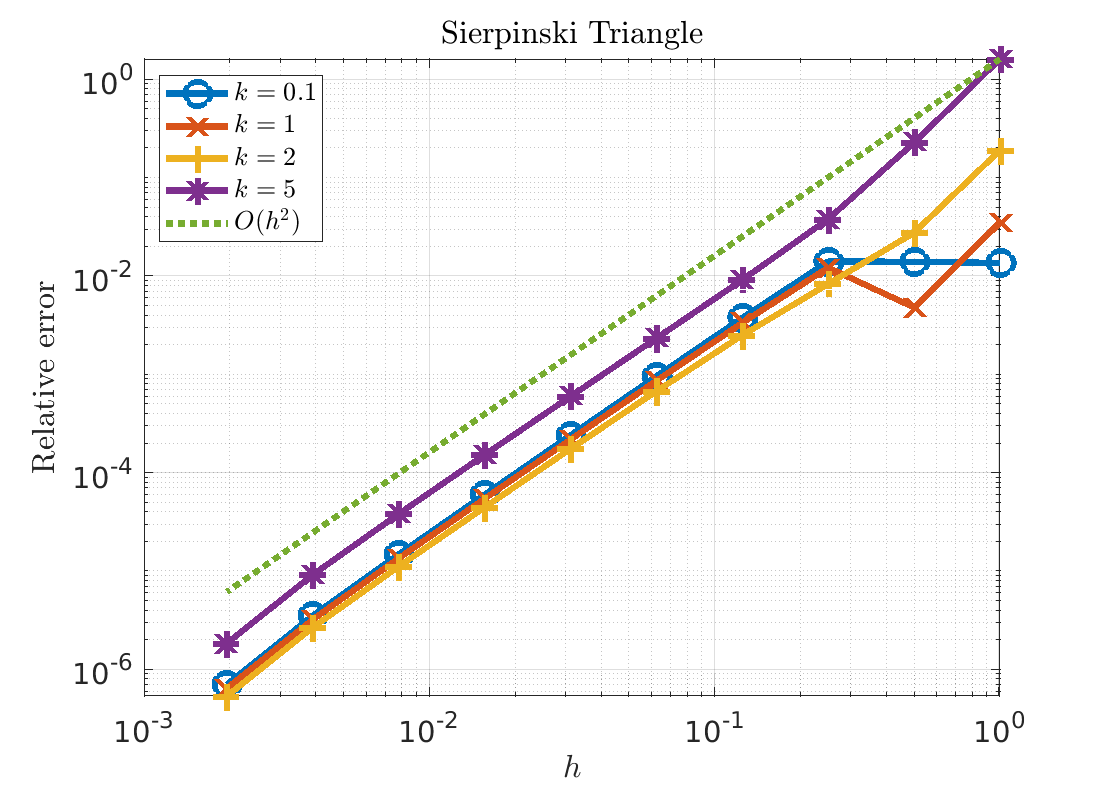}
	\includegraphics[width=0.49\linewidth]{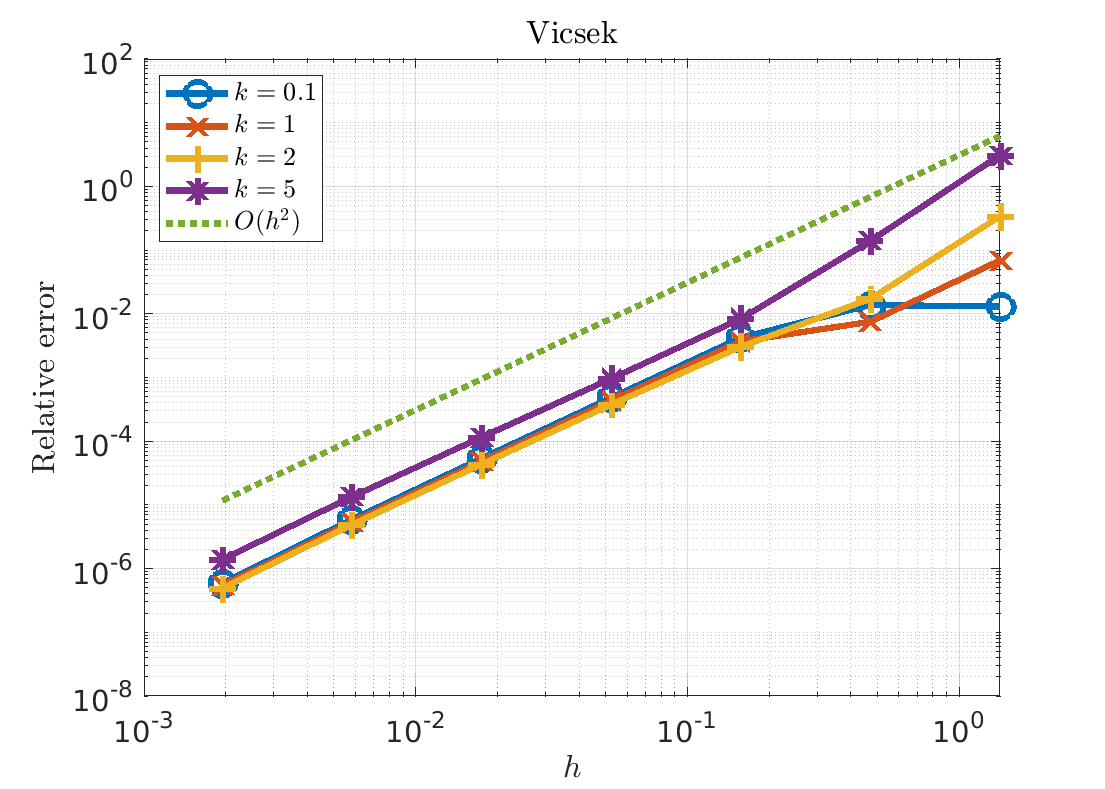}
	\includegraphics[width=0.49\linewidth]{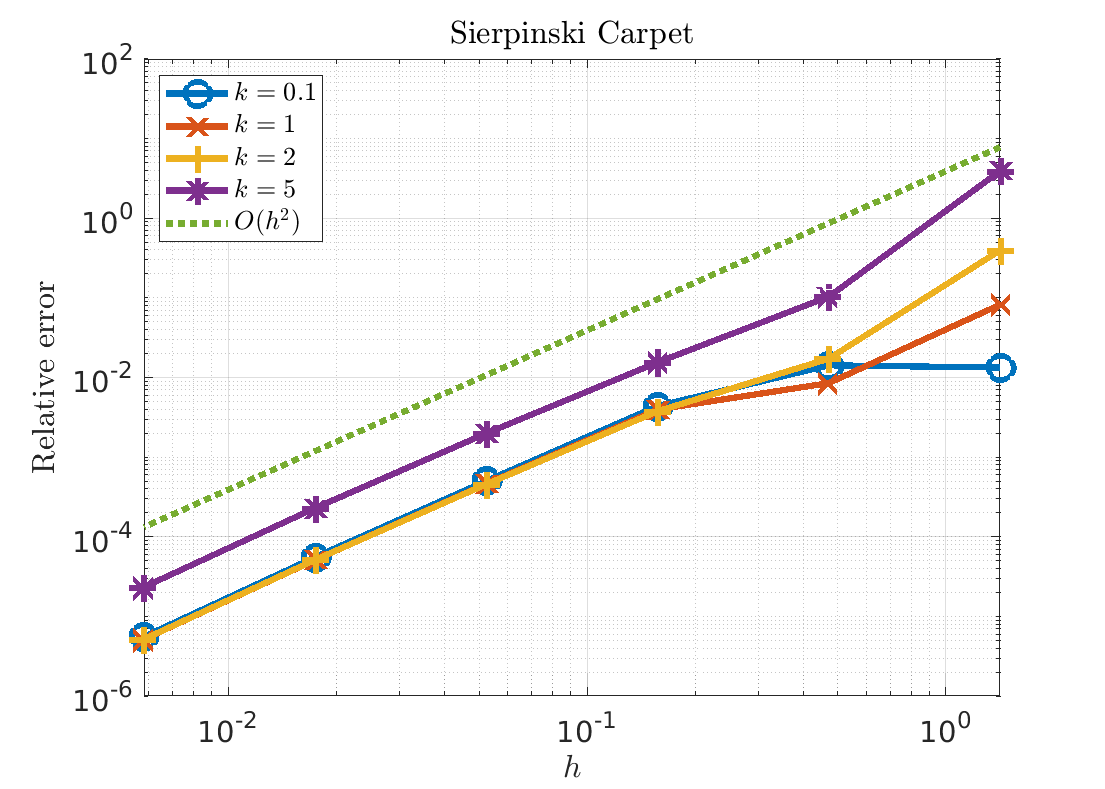}
	\includegraphics[width=0.49\linewidth]{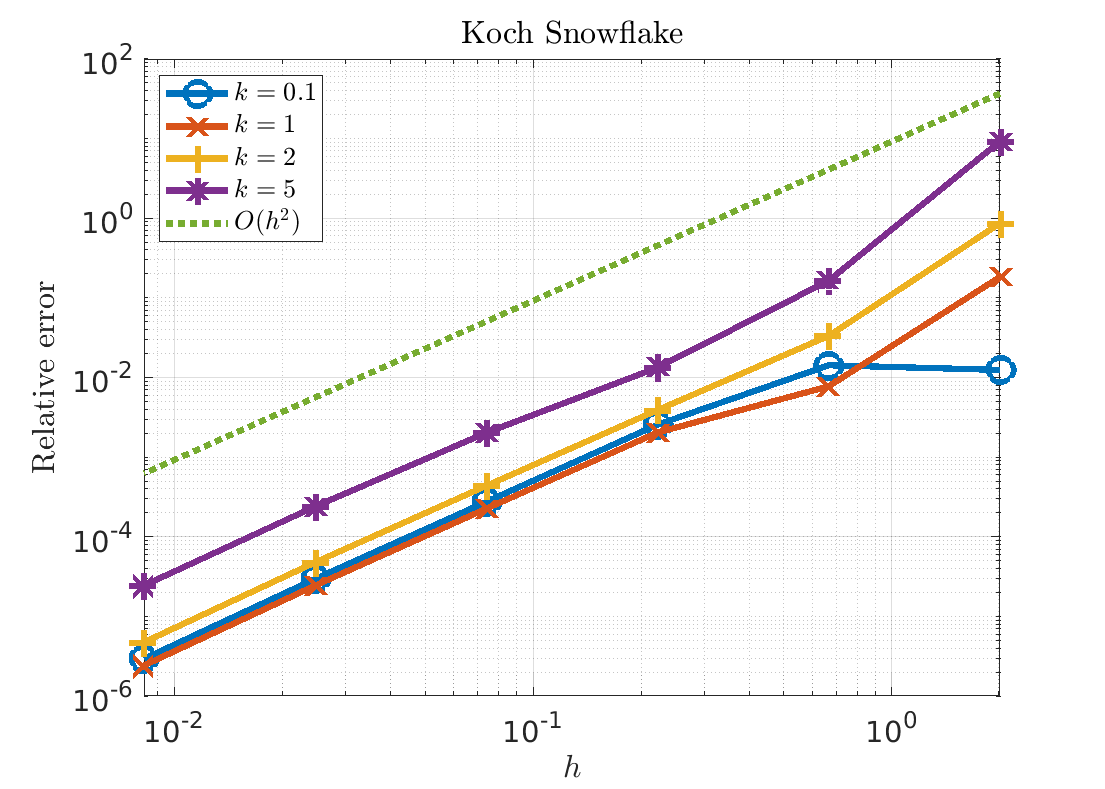}
	\caption{
Convergence of composite barycentre rule approximations to the second integral on the right-hand side of \eqref{eq:singsub} in the case $\Gamma_\bm=\Gamma_\bn=\Gamma$, as considered in \S\ref{sec:HausdorffBEM}.
	}
	\label{fig:singularity-subtraction}
\end{figure}

When $\Gamma_\bm$ and $\Gamma_\bn$ are disjoint, the integral \eqref{eq:GalInt} has a smooth integrand and can be evaluated using the composite barycentre rule with some maximum mesh width $h\leq h_{\rm BEM}$, with error $O(h^2)$ (by Theorem \ref{thm:Barycentre}(ii)). 
When $\Gamma_\bm\cap \Gamma_\bn$ is non-empty the integral \eqref{eq:GalInt} is singular, and to evaluate it we adopt a singularity subtraction approach, writing 
\begin{equation}\label{eq:singsub}
\int_{\Gamma_\bm}\int_{\Gamma_\bn}\Phi(x,y)~\rd\cH^d(y)\rd\cH^d(x) = \frac{1}{4\pi}\int_{\Gamma_\bm}\int_{\Gamma_\bn} \frac{1}{|x-y|}~\rd\cH^d(y)\rd\cH^d(x)+\int_{\Gamma_\bm}\int_{\Gamma_\bn}\Phi_*(x,y)~\rd\cH^d(y)\rd\cH^d(x),
\end{equation}
where $\Phi_*(x,y):=\Phi(x,y)-(4\pi|x-y|)^{-1} = (\re^{\ri k |x-y|}-1)/(4\pi|x-y|)$.
The first integral on the right-hand side of \eqref{eq:singsub} can be evaluated using the methods of this paper with $t=1$. In more detail, if $\Gamma_\bm = \Gamma_\bn$ this first integral will be similar to $I_{\Gamma,\Gamma}$, and if $\Gamma_\bm \neq \Gamma_\bn$ it will be similar to one of the other fundamental singular sub-integrals encountered in Algorithm \ref{alg:Algorithm}. In both cases it can be evaluated by combining Algorithm \ref{alg:Algorithm} with the composite barycentre rule, again with mesh width $h\leq h_{\rm BEM}$ and error $O(h^2)$.  
The second integral on the right-hand side of \eqref{eq:singsub} has a Lipschitz continuous integrand, 
and hence can be evaluated using the composite barycentre rule directly. According to  
Theorem \ref{thm:Barycentre}(i), the error in this approximation is guaranteed to be $O(h)$. 
In fact, for disjoint homogeneous attractors the error in evaluating this second term was proved in \cite[Proposition~5.5]{HausdorffQUAD} to be $O(h^2)${, and} experiments in \cite[Figure~8(a)]{HausdorffQUAD} suggest that the same may be true for certain non-homogeneous disjoint attractors. 
In Figure \ref{fig:singularity-subtraction} we present numerical results suggesting, furthermore, that the same may also be true for certain non-disjoint attractors. %
The plots in Figure \ref{fig:singularity-subtraction} show the relative error (against a high order reference solution) in computing the second term in \eqref{eq:singsub} using {the composite barycentre rule}, for the four attractors from \S\ref{sec:Sierpinski}-\S\ref{sec:Snowflake} and a range of wavenumbers, in the case where $\Gamma_\bm=\Gamma_\bn=\Gamma$. This case is chosen since it represents the most difficult case, in which $\Gamma_\bm$ and $\Gamma_\bn$ have full overlap. For all four examples we clearly observe $O(h^2)$ error in the numerical results. However, we leave theoretical justification of this observation to future work.

\begin{figure}[t!]
	\centering
	\includegraphics[width=0.48\linewidth]{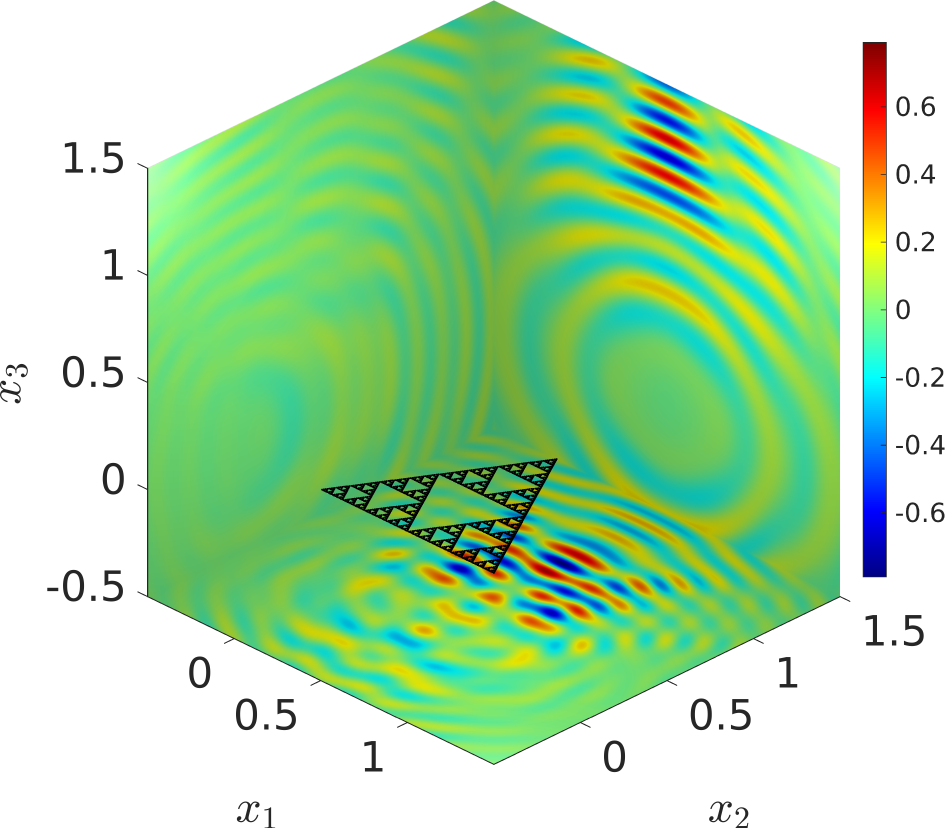}
	\hspace{2mm}
	\includegraphics[width=0.48\linewidth]{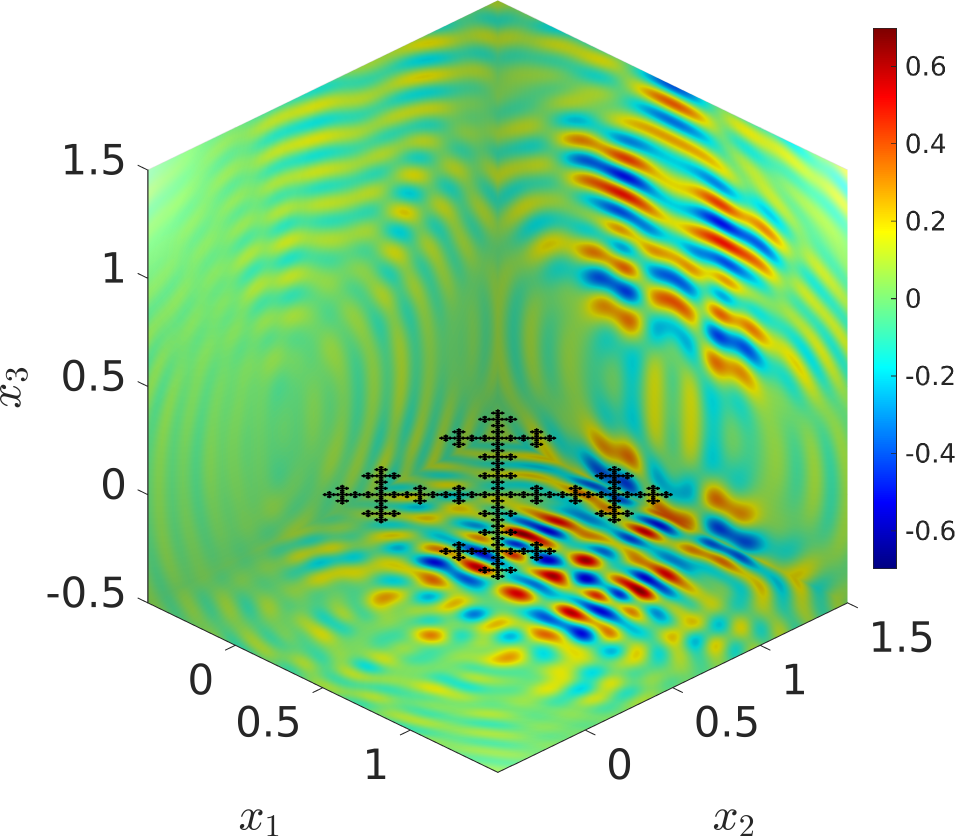}

\includegraphics[width=0.49\linewidth]{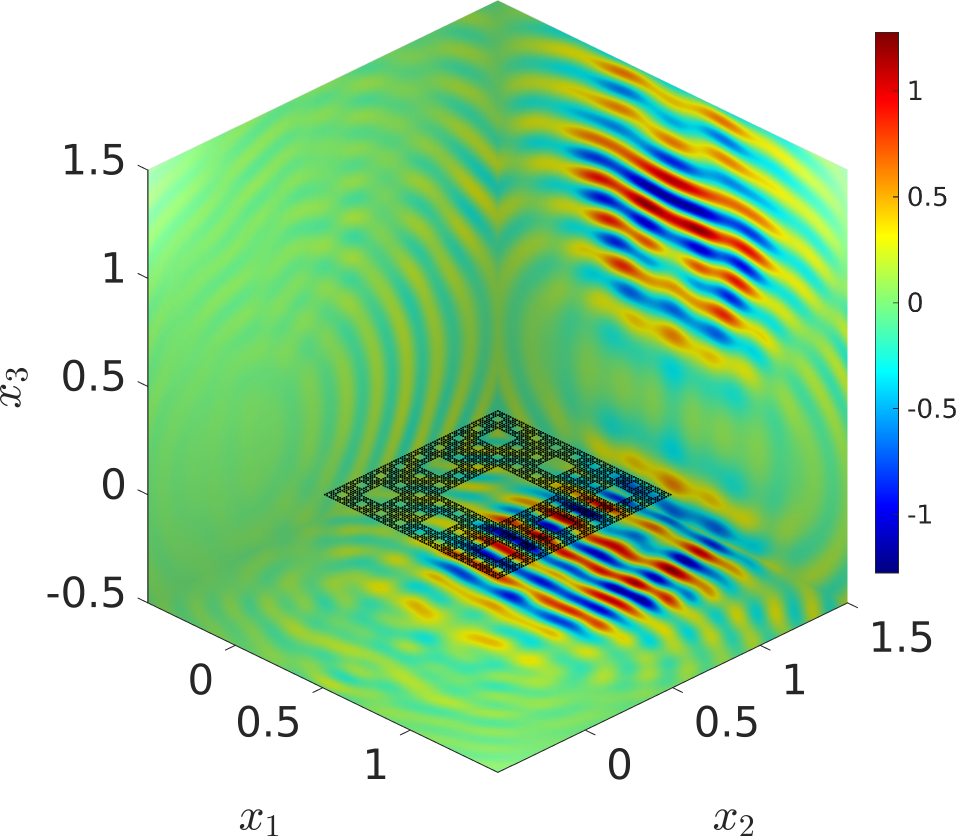}
	\hspace{2mm}
	\includegraphics[width=0.46\linewidth]{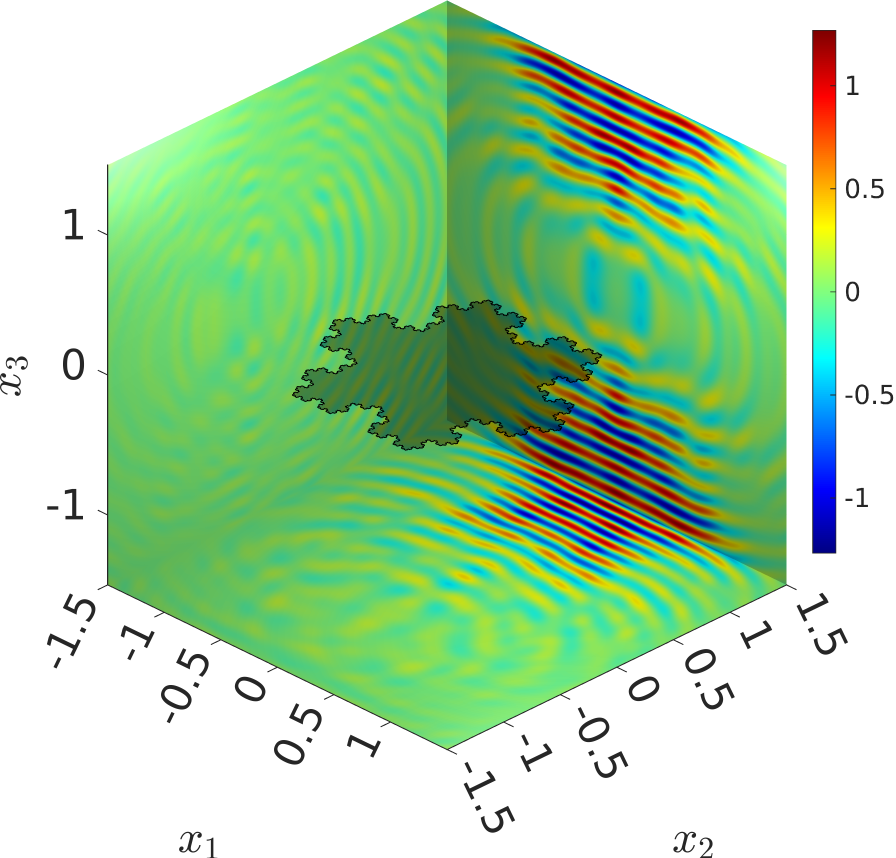}
	\caption{Field scattered by Sierpinski triangle (top left), Vicsek fractal (top right), Sierpinski carpet (bottom left) and Koch snowflake (bottom right), for the Dirichlet screen scattering problem described in \S\ref{sec:HausdorffBEM} with wavenumber $k=50$ and incident angle $\vartheta=(0,1,-1)/\sqrt{2}$. In each case the screen (sketched in black) lies in the plane $x_3=0$. 
}
	\label{fig:ScattFields}
\end{figure}

We end the paper by presenting in Figure \ref{fig:ScattFields} plots of the scattered field computed by our Hausdorff BEM solver (available at \url{www.github.com/AndrewGibbs/IFSintegrals}) for scattering by the four attractors from \S\ref{sec:Sierpinski}-\S\ref{sec:Snowflake}. In each case the wavenumber $k=50$ and incident angle $\vartheta=(0,1,-1)/\sqrt{2}$.
Here $h_{\mathrm{BEM}}=\diam(\Gamma)\rho^{\ell_{\mathrm{BEM}}}$, where $\rho$ is as defined in \S\ref{sec:nvo} and 
$\ell_{\mathrm{BEM}}=5$ for the Sierpinski triangle, $\ell_{\mathrm{BEM}}=4$ for the Vicsek fractal, 
$\ell_{\mathrm{BEM}}=4$ for the Sierpinski Carpet, 
and
$\ell_{\mathrm{BEM}}=8$ for the Koch snowflake, 
so that in each case we are discretising with at least 5 elements per wavelength. 
The Galerkin BEM matrix is constructed as described above, with $h=h_{\mathrm{BEM}}\rho^{4}$ in each case, taking advantage also of the \emph{reduced quadrature} approach described in \cite[Remark~5.19]{HausdorffBEM} (which exploits the far-field decay in $\Phi(x,y)$ to reduce the number of quadrature points for pairs of elements $\Gamma_\bm$ and $\Gamma_\bn$ that are well-separated).

We note that for disjoint attractors the Hausdorff BEM is supported by a fully discrete convergence analysis (presented in \cite{HausdorffBEM}). A similar analysis for the case $d=2$ (applying for instance to the Koch snowflake) will be presented in a forthcoming article \cite{CaChGiHe23}.  

\subsection*{Acknowledgements}
AG and DH acknowledge support from the EPSRC grant EP/V053868/1, and thank the Isaac Newton Institute for Mathematical Sciences, Cambridge, for support and hospitality during the programme \textit{Mathematical theory and applications of multiple wave scattering}, where work on this paper was undertaken. This work was supported by EPSRC grant no EP/R014604/1. 
BM and DH gratefully acknowledge support from the LMS Undergraduate Research Bursary scheme, which funded BM on a summer research internship at UCL, during which this work was initiated.

\bibliography{fractalQuad}

\begin{thebibliography}{10}

\bibitem{barnsley2013developments}
{\sc M.~Barnsley and A.~Vince}, {\em Developments in fractal geometry}, Bull.
  Math. Sci., 3 (2013), pp.~299--348.

\bibitem{barnsley1985iterated}
{\sc M.~F. Barnsley and S.~Demko}, {\em Iterated function systems and the
  global construction of fractals}, Proc. Roy. Soc. A. Math. Phys. Sci., 399
  (1985), pp.~243--275.

\bibitem{bessis1987mellin}
{\sc D.~Bessis, J.~Fournier, G.~Servizi, G.~Turchetti, and S.~Vaienti}, {\em
  Mellin transforms of correlation integrals and generalized dimension of
  strange sets}, Phys. Rev. A, 36 (1987), p.~920.

\bibitem{Bogachev}
{\sc V.~I. Bogachev}, {\em Measure Theory (Volume 1)}, Springer, 2007.

\bibitem{borm2005hierarchical}
{\sc S.~B{\"o}rm and W.~Hackbusch}, {\em Hierarchical quadrature for singular
  integrals}, Computing, 74 (2005), pp.~75--100.

\bibitem{HausdorffBEM}
{\sc A.~M. Caetano, S.~N. Chandler-Wilde, A.~Gibbs, D.~Hewett, and A.~Moiola},
  {\em {A Hausdorff measure boundary element method for acoustic scattering by
  fractal screens}}, arxiv preprint 2212.06594,  (2022).

\bibitem{CaChGiHe23}
{\sc A.~M. Caetano, S.~N. Chandler-Wilde, A.~Gibbs, and D.~P. Hewett}, {\em
  Properties of {IFS} attractors with non-empty interiors and associated
  function spaces and scattering problems}, In preparation.

\bibitem{CaCo:09}
{\sc F.~Calabr\`o and A.~Corbo~Esposito}, {\em An evaluation of
  {C}lenshaw-{C}urtis quadrature rule for integration w.r.t. singular
  measures}, J. Comput. Appl. Math., 229 (2009), pp.~120--128.

\bibitem{Fal}
{\sc K.~Falconer}, {\em Fractal Geometry: Mathematical Foundations and
  Applications}, Wiley, 3rd ed., 2014.

\bibitem{forte1998chaos}
{\sc B.~Forte, F.~Mendivil, and E.~Vrscay}, {\em {``Chaos games'' for iterated
  function systems with grey level maps}}, SIAM J. Math. Anal., 29 (1998),
  pp.~878--890.

\bibitem{Ga:90}
{\sc W.~Gautschi}, {\em Computational aspects of orthogonal polynomials}, in
  Orthogonal polynomials: Theory and Practice, P.~Nevai, ed., NATO Adv. Sci.
  Inst. Ser. C Math. Phys. Sci., Kluwer Acad. Publ., Dordrecht, 1990,
  pp.~181--216.

\bibitem{Ga:04}
\leavevmode\vrule height 2pt depth -1.6pt width 23pt, {\em Orthogonal
  polynomials: computation and approximation}, OUP, 2004.

\bibitem{HausdorffQUAD}
{\sc A.~Gibbs, D.~Hewett, and A.~Moiola}, {\em Numerical quadrature for
  singular integrals on fractals}, Numer. Algorithms, 92 (2023),
  pp.~2071--2124.

\bibitem{HaTo:13}
{\sc N.~Hale and A.~Townsend}, {\em Fast and accurate computation of
  {G}auss-{L}egendre and {G}auss-{J}acobi quadrature nodes and weights}, SIAM
  J. Sci. Comput., 35 (2013), pp.~A652--A674.

\bibitem{hutchinson1981fractals}
{\sc J.~E. Hutchinson}, {\em Fractals and self-similarity}, Indiana Univ. Math.
  J., 30 (1981), pp.~713--747.

\bibitem{kunze2011fractal}
{\sc H.~Kunze, D.~La~Torre, F.~Mendivil, and E.~R. Vrscay}, {\em Fractal-based
  Methods in Analysis}, Springer, 2011.

\bibitem{Ma:96}
{\sc G.~Mantica}, {\em {A stable Stieltjes technique for computing orthogonal
  polynomials and Jacobi matrices associated with a class of singular
  measures}}, Constr. Approx., 12 (1996), pp.~509--530.

\bibitem{Ma:00}
{\sc G.~Mantica}, {\em On computing {J}acobi matrices associated with recurrent
  and {M}\"{o}bius iterated function systems}, in Proceedings of the 8th
  {I}nternational {C}ongress on {C}omputational and {A}pplied {M}athematics,
  {ICCAM}-98 ({L}euven), vol.~115(1-2), 2000, pp.~419--431.

\bibitem{mantica2007asymptotic}
{\sc G.~Mantica and S.~Vaienti}, {\em {The asymptotic behaviour of the Fourier
  transforms of orthogonal polynomials I: Mellin transform techniques}}, Ann.
  Henri Poincar{\'e}, 8 (2007), pp.~265--300.

\bibitem{Mattila15}
{\sc P.~Mattila}, {\em Fourier Analysis and {H}ausdorff Dimension}, CUP, 2015.

\bibitem{meszmer2010hierarchical}
{\sc P.~Meszmer}, {\em Hierarchical quadrature for multidimensional singular
  integrals}, J. Numer. Math., 18 (2010), pp.~91--117.

\bibitem{meszmer2012hierarchical}
\leavevmode\vrule height 2pt depth -1.6pt width 23pt, {\em Hierarchical
  quadrature for multidimensional singular integrals - part ii}, J. Numer.
  Math., 22 (2014), pp.~33--60.

\bibitem{moran1998singularity}
{\sc M.~Mor{\'a}n and J.-M. Rey}, {\em {Singularity of self-similar measures
  with respect to Hausdorff measures}}, T. Am. Math. Soc., 350 (1998),
  pp.~2297--2310.

\bibitem{strichartz1990self}
{\sc R.~S. Strichartz}, {\em {Self-similar measures and their Fourier
  transforms I}}, Indiana U. Math. J.,  (1990), pp.~797--817.

\bibitem{strichartz2000evaluating}
\leavevmode\vrule height 2pt depth -1.6pt width 23pt, {\em Evaluating integrals
  using self-similarity}, The American Mathematical Monthly, 107 (2000),
  pp.~316--326.

\bibitem{ToTrOl:16}
{\sc A.~Townsend, T.~Trogdon, and S.~Olver}, {\em Fast computation of {G}auss
  quadrature nodes and weights on the whole real line}, IMA J. Numer. Anal., 36
  (2016), pp.~337--358.

\bibitem{Tr:08}
{\sc L.~N. Trefethen}, {\em Is {G}auss quadrature better than
  {C}lenshaw-{C}urtis?}, SIAM Rev., 50 (2008), pp.~67--87.

\bibitem{trefethen2011ten}
\leavevmode\vrule height 2pt depth -1.6pt width 23pt, {\em Ten digit problems},
  in An Invitation to Mathematics: from Competitions to Research, D.~Schleicher
  and M.~Lackmann, eds., Springer, 2011, pp.~119--136.

\bibitem{Tr:13}
\leavevmode\vrule height 2pt depth -1.6pt width 23pt, {\em Approximation theory
  and approximation practice}, SIAM, 2013.

\end{thebibliography}
\bibliographystyle{siam}
\end{document}